\documentclass[11pt]{amsart}
\usepackage[utf8]{inputenc}
\usepackage{enumitem}
\usepackage{amssymb}
\usepackage{mathtools}
\usepackage{amsthm}
\usepackage{amsmath}
\usepackage[dvipsnames]{xcolor}
\usepackage{mathtools}
\usepackage{tikz}
\usetikzlibrary{arrows,decorations.pathmorphing,backgrounds,positioning,fit,decorations.pathreplacing}
\usepackage{enumitem}
\usepackage{dsfont}
\usepackage{fullpage}
\usepackage{graphicx}
\usepackage{pdfsync}

\usepackage[T1]{fontenc}

\usepackage{hyperref}
 \hypersetup{
     colorlinks=true,
     linktocpage=true,
     linkcolor=red,
     filecolor=blue,
     citecolor = blue,
     urlcolor=cyan,
     }

\usepackage[a4paper, twoside=false, vmargin={2cm,3cm}, includehead]{geometry}

\usepackage{comment}
\usepackage[capitalize]{cleveref}
\usepackage{standalone}
\usepackage{url}
\usepackage{float}

\theoremstyle{plain}
\newtheorem{theorem}{Theorem}[section] 
\newtheorem{proposition}[theorem]{Proposition}
\newtheorem{lemma}[theorem]{Lemma}
\newtheorem{corollary}[theorem]{Corollary}
\theoremstyle{definition}
\newtheorem{remark}[theorem]{Remark}
\newtheorem{definition}[theorem]{Definition} 
\newtheorem{example}[theorem]{Example} 

\newtheorem{claim}[theorem]{Claim} 

\theoremstyle{plain}
\newtheorem{thmx}{Theorem}

\newcommand{\Cercle}{\mathbb{S}}
\newcommand{\N}{\mathbb{N}}
\newcommand{\Z}{\mathbb{Z}}

\newcommand{\F}{\mathbb{F}}
\newcommand{\Q}{\mathbb{Q}}
\newcommand{\R}{\mathbb{R}}

\newcommand{\x}{\mathtt{x}}

\newcommand{\factor}{\sqsubseteq}

\newcommand{\ND}{\textnormal{ND}}

\newcommand{\Aut}{\textnormal{Aut}}

\newcommand{\Cone}{\textnormal{HB}}

\newcommand{\topo}{\textnormal{top}}

\newcommand{\Conv}{\textnormal{Conv}}
\newcommand{\dist}{\textnormal{dist}}
\newcommand{\red}{\textnormal{red}}
\newcommand{\HB}{\textnormal{HB}}
\newcommand{\ve}{\varepsilon}

\renewcommand{\tt}[1]{\mathtt{#1}}

\renewcommand{\b}[1]{{\bf #1}}

\newcommand\numberthis{\addtocounter{equation}{1}\tag{\theequation}}

\newcommand{\act}{\curvearrowright}

\newcommand{\edit}[3]{\color{#1}{#3}\color{black}\marginpar{\textcolor{#1}{[[#2]]}}}
\newcommand{\sam}[1]{\edit{blue!50}{SP}{#1}}
\newcommand{\sd}[1]{\edit{green!60!black}{SD}{#1}}

\newcommand{\define}[1]{{\em #1}}

\author{Nicol\'as Bitar, Sebasti\'an Donoso, Samuel Petite}

\address[Nicol\'as Bitar]{Laboratoire Ami\'enois
	de Math\'ematiques Fondamentales et Appliqu\'ees, CNRS-UMR 7352, Universit\'{e} de Picardie Jules Verne, 33 rue Saint Leu, 80039   Amiens cedex 1,
	France.}
\email{nicolas.bitar@u-picardie.fr}

\address[Sebasti{\'a}n Donoso]{Departamento de Ingenier\'{\i}a Matem\'atica and Centro de Modelamiento Matem{\'a}tico, Universidad de Chile \& IRL 2807 - CNRS, Beauchef 851, Santiago, Chile} \email{sdonosof@uchile.cl}

\address[Samuel Petite]{Laboratoire Ami\'enois
	de Math\'ematiques Fondamentales et Appliqu\'ees, CNRS-UMR 7352, Universit\'{e} de Picardie Jules Verne, 33 rue Saint Leu, 80039   Amiens cedex 1,
	France.} \email{samuel.petite@u-picardie.fr}

\thanks{S. Petite and N. Bitar were supported by the ANR project IZES ANR-22-CE40-0011. S. Donoso was partially funded by ANID/Fondecyt/Regular 1241346 and by ANID PIA/BASAL FB210005 (Centro de Modelamiento Matemático). The three authors thank the support of the project MathAmSud240026, which funded research visits that made this collaboration possible.}


\title{Non-determinism in group actions and topological minimal self-joinings}

\begin{document}

\begin{abstract}
We study the topological and geometrical aspects of non-determinism in group actions, generalizing the notion of non-expansive direction for abelian actions. Under suitable conditions on the acting group, we establish restrictions on the collection of non-deterministic horoballs (or half-spaces) for any action. We illustrate our results in the case of free and nilpotent groups, expanding upon known cases for abelian groups. We then apply these findings to standing problems concerning minimal self-joinings in topological dynamics, which are of independent interest.
\end{abstract}

\maketitle 

\section{Introduction} 

A topological dynamical system consists of a compact phase space $X$ endowed with an action of a group $G$ by homeomorphisms. For these systems, analyzing  pairs of points  that remain ``indistinguishable'' or ``asymptotic'' (up to a given precision) along specific  orbit subsets is  a foundational challenge. This problem  is deeply connected to classical dynamical concepts such as expansiveness, complexity, and entropy. 

For $\mathbb{Z}$-actions, a seminal result by S. Schwartzman~\cite{Schwartzman_thesis:1953} states that in any infinite compact metric space there exist distinct points that stay arbitrarily close for all positive (or negative) times. Later, motivated by the study of higher-rank abelian actions, M. Boyle and D. Lind generalized this to $\mathbb{Z}^d$-actions, introducing the concept of non-expansive half-spaces. In \cite{Boyle_Lind_expansive_subdynamics:1997}, they showed that there always exists a half-space in $\mathbb{R}^d$, such that for any $\epsilon>0$, there exist two distinct points that stay at distance at most $\epsilon$ along iterations by elements within that half-space. Non-deterministic subspaces and asymptotic pairs play key roles in studying expansive maps and dimension \cite{Mane_expansive_and_dimension:1979}, topological minimal self-joinings~\cite{King_top_minimal_self_joinings:1990}, directional entropy~\cite{Park_directional_entropy:1999}, combinatorial problems such as Nivat's conjecture~\cite{Cyr_Kra_nonexp_Nivat:2015,kari2023decidability}, and even the construction of strongly aperiodic subshifts of finite type on some groups~\cite{Aubrun_Bitar_Huriot-Tattegrain_strong_SFT_BS_groups:2024}.

Because the definitions and proofs for these classical results relied heavily on the linear structure of $\mathbb{R}^d$, extending them to general group actions remained an open challenge for several years. Recently, Donoso, Maass, and Petite introduced in~\cite{Donoso_Maass_Petite_geometric_asymptotic:2024} a novel geometric framework to study this phenomenon, which is referred to as non-determinism, for groups beyond $\mathbb{Z}^d$. 
This extension was introduced for all second countable groups that admit a right invariant and proper metric. In this case, horoballs (sublevel sets of horofunctions) play the role of half-spaces. One of their main contributions is a result that simultaneously unifies Schwartzman's and Boyle-Lind's theorems, extending them to general group actions (\cref{thm:directed_RC}).  We defer the precise formulation of this result to~\cref{subsec:non-det_topo}.
It is worth noting that while \cite{Donoso_Maass_Petite_geometric_asymptotic:2024} established several results for all groups admitting a proper right-invariant metric, the consequences and the identification of non-deter\-mini\-stic horo\-balls were mostly confined to the abelian setting.

In this paper, we continue the research program initiated in \cite{Donoso_Maass_Petite_geometric_asymptotic:2024} by exploring deeper topological and geometrical aspects of non-determinism, particularly in non-abelian groups. As a consequence, we get tools to study topological minimal self-joinings. We exhibit rigidity constraints, showing that the rank of the acting group bounds the degree of topological minimal self-joinings. 

\subsection*{The topology of non-deterministic horofunctions}
To investigate non-determinism in general group actions, we examine the structure and topology of the set of non-deterministic horofunctions. 
A horofunction is a continuous, real-valued function on the metric group $G$. Geometrically, it captures the asymptotic behavior of sequences escaping to infinity—a fundamental concept in geometric group theory.
We say a horofunction $h$ is $\varepsilon$-non-deterministic if there exist two distinct points in the space that remain $\varepsilon$-close under the action of every element within the horoball $\{h < 0\}$. We denote the collection of all such horofunctions for a given system by $\ND_\varepsilon(X)$, and $\ND(X)=\bigcap_{\ve>0}\ND_{\ve}(X)$. 
An important feature of Boyle and Lind's result in $\Z^d$ is that there exists a half-space for which there exist $\ve$-asymptotic points for every $\ve>0$. That is, the half-space is independent of $\ve$. This property is very important for many of the applications this concept has found.  For the more general statement in \cite{Donoso_Maass_Petite_geometric_asymptotic:2024}, the horofunction may depend on $\ve$. For this reason, it is worth studying under which conditions the space $\ND_\ve(X)$ is closed, which guarantees that we can find a horofunction that is $\ve$-non-deterministic for all $\ve$ (that is, $\ND(X)\neq \emptyset$). By imposing additional conditions on the space of horofunctions of the group (\Cref{def:qsubadditive}), we establish the following sufficient condition.

\begin{thmx}\label{thm:A}
Let $G$ be an infinite group with a proper right-invariant metric. Then for any $\ve>0$, the set of non-deterministic horofunctions  $\ND_\ve(X)$ is non-empty for any infinite system $(X, T, G)$. 

If furthermore $G$ has uniformly quasi-subadditive horofunctions and no lower bound for the horofunctions, then  the sets  $\ND_\ve(X)$ and $\ND(X)$ are non-empty  closed sets.
\end{thmx}

The group $\Z^d$ equipped with  the $\ell^2$ metric (or euclidean distance) fulfills the hypothesis of \cref{thm:A}, so that it recovers  the  results of Boyle and Lind \cite{Boyle_Lind_expansive_subdynamics:1997}. This is also the case for the Heisenberg group equipped with the Kor\'anyi metric (see \cref{ex:heisenberg}).
It is worth noting that the $\ell^1$ metric (or word metric) on $\Z^d$ does not satisfy the hypotheses of this theorem (more precisely, the quasi-subadditive condition). Indeed, under such a metric, $\ND_{\varepsilon}(X)$ may fail to be closed (as shown by examples in \cite{Donoso_Maass_Petite_geometric_asymptotic:2024}). However, richer geometric structures do allow these conditions to be met, and we believe they provide the natural framework for investigating these phenomena, like for instance for Carnot groups.

In fact, we provide a more general statement than \cref{thm:A}, in the form of~\cref{thm:closeness-ND}, which is localized and relative to a factor map. Together with a former result from \cite{Donoso_Maass_Petite_geometric_asymptotic:2024} (see \cref{thm:robinson_crusoe}) we prove \cref{thm:A}.\\

Next, we examine structural constraints on non-deterministic horoballs by considering their intersections. Once again, under some assumptions on the boundary (\Cref{def:SymHoroball}), we establish the following.

\begin{thmx}\label{thm:B}
   Let $G$ be a discrete group that admits a metric with symmetric horoballs space, and let $(X,T,G)$ be an infinite topological dynamical system such that $\ND_{\varepsilon}(X)$ is closed for every $\ve>0$. 
Then,
    \[\bigcap_{h\in\ND(X)} \{h<0\} \subseteq {\rm Tor}(G),\]
where $\textnormal{Tor}(G)$ is the subset of torsion elements of $G$. In particular, if the group $G$ is torsion-free, the above intersection is empty.
\end{thmx}
Of course \cref{thm:B} is meaningful when the horofunctions of $G$ do take negative values, which occurs, for instance, when $G$ is finitely generated (see \cref{lem:thick-horoballs-fg}).  Again, as for the former result, we state \cref{thm:vacia_tf} that is a more general  version of \cref{thm:B} relative to a factor map.

To illustrate the breadth of the theory, we explicitly develop these geometric concepts for free and nilpotent groups (specifically discrete Heisenberg groups), extending phenomena previously understood only for abelian groups. 
In particular  for the discrete Heisenberg group, notice  that it admits a degenerate horofunction— a zero constant function— which provides  a trivial horoball. Hence, the conclusion of \cref{thm:B} may fail to give meaningful insight.
To address this issue, we introduce a refined and stronger version of the result (\cref{thm:vacia_th_Heisenberg}), where the same conclusion as in  \cref{thm:B} holds when restricted to the intersection of all non-deterministic Busemann horofunctions (i.e. horofunctions arising from geodesics). In particular  such (at least two) non-degenerate and non-deterministic horofunctions always exist.

\subsection*{Topological minimal self-joinings and their restrictions}
The second half of the article is dedicated to strong topological minimal self-joinings (sTMSJ). This notion was introduced by Del Junco~\cite{delJunco_minimal_self_joining_top_dyn:1987} as a topological dynamical analog of the notion of minimal joinings of Rudolph from measurable dynamics~\cite{Rudolph_example_mpm_min_self_joinings:1979}. Their link to non-expansive subspaces is due to King~\cite{King_top_minimal_self_joinings:1990}, who used them to show that no $4$-fold sTMSJ exist for $\mathbb{Z}$-actions. 

We start with the study of strong doubly minimal systems, which correspond to $2$-fold sTMSJ. 
We give an example of a 2-fold sTMSJ  $\Z^d$ example (\cref{theo:KingExample}). 
We establish generalizations of their properties from their $\Z$ counterparts (\Cref{prop:expansive}), and also find novel behavior such as strong doubly minimal actions on connected spaces (Example~\ref{ex:dm_connected}).
We summarize the results in the following theorem
\begin{thmx} \label{thm:summarize_doubly_minimal}
    Consider a countable group $G$ and let $(X,T,G)$ be a strong doubly minimal system, then
    \begin{itemize}
        \item Its automorphism group consists of maps $T^g$, with $g\in Z(G)$.
        \item The system is a minimal zero-entropy expansive system.
        \item The system is prime (i.e. admits no non trivial factor) when $G$ is finitely generated and its center $Z(G)$ is torsion free and acts faithfully. 
    \end{itemize}
\end{thmx}

\cref{thm:summarize_doubly_minimal} follows combining \cref{prop:factor_trivial,prop:expansive,lem:min+faithfull=>free,prop:doublymin_zero_entropy}.

As an application of our theorems for non-deterministic horofunctions, and in the spirit of King's limit to foldings, we obtain new results of this type for various classes of groups. Through a black-box theorem that limits folds according to covering of the group in terms of horofunctions (\Cref{thm:king_general}), we establish limits for some abelian groups and Heisenberg groups.

\begin{thmx} \label{thm:no-tmjs-Zd}
    Let $G$ be a finitely generated abelian group.  Then, no system $(X,T,G)$ has $2(d+1)$-fold sTMSJ where $d$ is the rank of $G$.
\end{thmx}

\begin{thmx}\label{thm:no-tmjs-Heisenberg}
Let $H_{2d+1}(\Z)$ be a discrete Heisenberg group. For every $d\geq 1$, no system $(X,T,H_{2d+1}(\Z))$ has $4d+2$-fold sTMSJ.
\end{thmx}

In contrast to discrete actions, $\R$-flows present different behaviors. For instance, Ratner's theorems imply that the horocycle flow on specific compact manifolds has $n$-fold sTMSJ for any integer $n>0$ (see  \cite{KanigowskiKasprzakLorenzo}).

\subsection*{Organization of the paper}
In \cref{sec:background}, we recall the necessary background on groups, horofunctions, and non-determinism in topological dynamics. \cref{sec:TopoHorofct} establishes the general topological properties of the set of non-deterministic horofunctions and in particular Theorems \ref{thm:A} and \ref{thm:B}. We give a geometric criterion for the closedness of non-deterministic horofunctions and explore structural constraints on their intersection and covering properties.  In \cref{sec:horo-heisenberg}, we analyze the horofunctions on the Heisenberg group equipped with the Kor\'anyi metric. To strengthen~\cref{thm:B} in this context, and thus establish~\cref{thm:vacia_th_Heisenberg}, we refine the technical result known as the directed Robinson Crusoe’s theorem (introduced in \cite{Donoso_Maass_Petite_geometric_asymptotic:2024}) and adapt the proof strategy used for Theorem \ref{thm:B}.
Finally, dynamical applications are given in \cref{sec:TMSJ}, specifically addressing strong doubly minimal systems and establishing new restrictions on the existence of $n$-fold topological minimal self-joinings (theorems \ref{thm:summarize_doubly_minimal}, \ref{thm:no-tmjs-Zd} and \ref{thm:no-tmjs-Heisenberg}). 

\subsection*{Acknowledgements}
We are deeply grateful to M. Hochman for illuminating discussions concerning the material in \cref{sec:TMSJ}, and for bringing the problems in \cref{sec:LimitnFolding} to our attention.

\section{Preliminaries and Background} \label{sec:background}

\subsection{Groups and their horofunctions}\label{sec:HorofctGroup} 
In this section, we recall several results concerning horofunctions and horoballs, to establish the framework for our main theorems.

Given a group $G$ we let $1_G$ denote its neutral element. Throughout what follows, we assume that $G$ is countable and admits a metric $\rho \colon G \times G \to \mathbb{R}$ satisfying the following two properties:
\begin{itemize}
\item it is right invariant, that is, $\rho(gf, kf) = \rho(g,k)$ for all $g,k,f \in G$;
\item it is proper, meaning that every closed ball is compact. 
\end{itemize}

Throughout this article, we restrict our attention to discrete groups. In this setting, properness means that every closed ball is actually finite. Furthermore, note that an element $k \in G$ is a torsion element if and only if the set $\{k^n : n \in \Z\}$ is bounded.

Generalizing the ideas of Busemann, Gromov defined a compactification of the group $G$ with an embedding map $b$. Denoting by $C(G)$ the collection of continuous real functions on $G$, this embedding is given by: 
\begin{eqnarray*}
b\colon G & \to & C(G)\\
g & \mapsto& b_{g} \colon x \mapsto \rho(g,x) -\rho(g,1_G). 
\end{eqnarray*}
The triangle inequality implies that all the maps $b_{g}$ are 1-Lipschitz. Moreover, by construction, we have  $b_{g}(1_G) =0$. It follows from  Arzel\`a-Ascoli's theorem and a standard diagonal argument that $b(G)$ is a relatively  compact set in $C(G)$ for the compact open topology. It is also straightforward to verify that the map $b$ is injective. 
The \define{border}  of $(G,\rho)$, denoted by $\partial (G, \rho)$ or simply  $\partial G$ when the metric is implied, is defined as the set 
$$ \partial G = \overline{b(G)} \setminus b(G),$$
where the closure is taken with respect to the compact open topology. This set is not empty whenever $G$ is unbounded. 
A function $h \in \partial G$ is called a \define{horofunction}. 
The \define{horoball} associated with a horofunction $h \in \partial G$ is the subset $H$ of $G$ given by
\[ H= \{x \in G\mid \ h(x)< 0\}.\]
For brevity, we may denote this set by $\{h<0\}$ and refer to it simply as a horoball. We refer to \cite{Bridson_Haefliger_metric_non-pos_curvature:1999} for discussion of several properties of horofunctions and horoballs.

A standard computation shows that the horofunctions of $\Z^{d}$ equipped with the Euclidean metric (or $\ell^{2}$ metric) are normalized linear functions: they are of the form $\langle \cdot,  \bf u \rangle$ for some unit vector $\bf u \in \R^{d}$. In particular, the associated horoball is an open half-space where $\bf u$ is the normal outgoing unit vector. For the $\ell^{1}$ metric (or word metric) on $\Z^{d}$, the horofunctions are piecewise linear and the horoballs are quarter or half-spaces with restricted directions.\\

\subsubsection{The no lower bound on horofunctions property}\label{sec:NoLowerBound}
In Section \ref{sec:closednessND} we will use the following notion. Let $\rho$ be a proper right-invariant metric on the group $G$. We say that  a closed subset $S \subseteq \partial G$  has \define{no lower bound for its horofunctions} if every horofunction $h \in S$ satisfies $\inf_{g\in G} h(g)= - \infty$. 
When $S = \partial G$ we simply say that $G$ has no lower bound for its horofunctions.
The compactness of the boundary $\partial G$ allows us to obtain a lower bound that goes to $-\infty$ uniformly for all horofunctions.
\begin{lemma}\label{lem:UnifLowerBound}
  If the closed set $S\subseteq \partial G$ has no lower bound for its horofunctions, then there is a uniform lower bound
   $$ \lim_{L\to \infty} \sup_{h \in S} \inf(h(B_L(1_G))) = - \infty.$$
\end{lemma}    
\begin{proof} For any $g\in G $ and $R>0$, set $U_g = \{ h \in S \mid h(g) <-R\}$. Then $U_g$ is an open set and $S = \bigcup_{g\in G} U_g$ because $S$ has no lower bound for its horofunctions. By compactness of $\partial G$, there are finitely many $g_1, \ldots, g_\ell \in G$ such that $S = \bigcup_{i=1}^\ell U_{g_i}$. Take $L>0$ such that $g_1,\ldots,g_{\ell}$ belong to $B_{L}(1_G)$. Then, for any horofunction $h$ in $S$ we have $\inf(h(B_L(1_G)))<-R$.  The conclusion follows.
\end{proof}

\begin{remark} \label{rmk:finitary} 
The conclusion of \cref{lem:UnifLowerBound} can be formulated in finitary terms using the group's metric as follows: for any $R>0$ there exists $L>0$ such that for all sufficiently large $g\in G$, one has \[ \min_{f\in B_L(1_G)} \rho(f,g) +R \leq \rho(1_G,g). \]
\end{remark}

A finitely generated group endowed with the $\ell^1$ (or word) metric has no lower bound for its horofunctions (see \cite[Proposition 1.5]{Auslander_Glasner_Weiss_recurrencer_zerodimensional:2007} or \cite[Proposition 2.1]{Donoso_Maass_Petite_geometric_asymptotic:2024}). This is also the case for $\Z^{d}$ equipped with the Euclidean norm, and for the discrete Heisenberg group with an explicit, geometrically meaningful metric (see \cref{ex:heisenberg}). 
A strategy for obtaining a metric with such a property consists of finding an embedding of the group into a Lie group with good geometric properties that ensure the conditions are satisfied. The reason we do not restrict ourselves to the word metric, and instead look for more geometrical metrics, is that under the word metric $\ND_{\ve}(X)$ may fail to be closed. As we shall discuss in \cref{sec:closednessND}, closedness is a useful property and we will show implied by suitable geometric conditions on the metric. \\

A subset $G'\subseteq G$ is \define{thick} if it contains arbitrarily large balls \footnote{The classical definition of thick set is slightly different (see for instance \cite[Chapter 1]{Glasner_ergodic_theory_joinings:2003}). However, in the context of a proper metric, this is equivalent to the definition given here.}. That is, for any $R>0$ there exists $g\in G'$ such that $B_{R}(g)\subseteq G'$. A subset $S\subseteq G$ is \define{(left) syndetic} if there exists a compact (in our setting, finite) set $K\subseteq G$ such that $KS=G$. The notions of thickness and syndeticity are ``dual'', in the sense that $G'\cap S\neq \emptyset$ for any thick subset $G'$ and any syndetic subset $S$.   

Let $G$ be a group with a right invariant and proper metric $\rho$. We say that $(G,\rho)$ has \define{thick horoballs} if $\{h<0\}$ is thick for any $h\in \partial G$.  We recall a result linking thickness and no lower bound for the horofunctions
\begin{proposition}[{\cite[Proposition 2.1]{Donoso_Maass_Petite_geometric_asymptotic:2024}}] \label{prop:nobound_imply_thick}
Let $G$ be a group with a right invariant and proper metric $\rho$. If $h\in \partial G$, satisfies $\inf_{g\in G} h(g)=-\infty$, then $\{h<0\}$ is thick. Consequently, if $(G, \rho)$ has no lower bound on its horofunctions, then it has thick horoballs.
\end{proposition}

We can characterize this property as follows:
\begin{lemma} \label{lem:thick-horoballs-fg}
Let $G$ be a countable group. The following are equivalent:
\begin{enumerate}
    \item $G$ is finitely generated.
    \item There exists a right-invariant and proper metric $\rho$ such that $(G,\rho)$ has thick horoballs.
    \item  There exists a right-invariant and proper metric $\rho$ such that each horoball is nonempty.
\end{enumerate}
    
\end{lemma}

\begin{proof}
$(1) \implies (2)$ is the content of \cite[Proposition 1.5]{Auslander_Glasner_Weiss_recurrencer_zerodimensional:2007} (see also \cite[Proposition 2.1]{Donoso_Maass_Petite_geometric_asymptotic:2024} for a strengthening). It is worth noting that in this case it suffices to take $\rho$ to be the word metric. As $(2)\implies (3)$ is obvious, we only justify $(3) \implies (1)$. 
Let $\rho$ be a right-invariant and proper metric such that each horoball is nonempty. For $g\in G$, set $U_g=\{h \in \partial G\mid h(g)<0\}$. Then $U_g$  is an open subset of $\partial G$ and since horoballs are not empty, we have $\partial G =\bigcup_{g\in G}U_g$. By compactness, there exist finitely many $g_1,\ldots,g_d$ such that $\partial G =\bigcup_{i=1}^\ell U_{g_i}$. We claim that $G$ is a finite extension of $A=\langle g_1,\ldots,g_\ell\rangle$, and is hence finitely generated. Assume for contradiction that it is not, and let $(f_i)_{i\in \mathbb{N}}$ be a sequence of representatives of cosets of $G/A$.  We may choose the representatives $f_i$ so that they satisfy 
\[\rho(f_i,1_G)\leq \inf_{\substack{f\in G\\ f_iA=fA}} \rho(f,1_G) + \frac1i,\]
that is, $f_i$ almost minimizes the distance to the identity within its coset. As $f_ig_k^{-1}$ is in the same coset as $f_i$, we get $\rho(f_{i}g_{k}^{-1},1_G)\geq \rho(f_i,1_G)-1/i$ , for all $k\in\{1,\ldots,\ell\}$. Since $G$ is countable and the metric is proper, we can pass to a subsequence and assume $\rho(f_i,1_G)\to \infty$.  Taking any limit of the $b_{f_i}$ gives a horofunction $h$ with $h(g_k)\geq 0$ for all $k$. A contradiction. 
\end{proof}

\begin{remark}
    The same proof above can be used in the non-discrete setting, yielding that $G/A$ is bounded. 
    Furthermore, it shows that if $G'\subseteq G$ is not of finite index, then there exists a horofunction $h$ such that $h(g)\geq 0$ for all $g\in G'$.   
    
Also note that being finitely generated is not enough to ensure the existence of thick horoballs for a proper-right invariant metric. For instance, in $\Z$, the metric defined by $\rho(n,m)=\sqrt{|n-m|}$ has empty horoballs. 
\end{remark}

\subsection{Non-determinism in topological dynamics}
\label{subsec:non-det_topo}

A topological dynamical system is a tuple $(X,T,G)$ where $X$ is a compact metric space, and $T\colon G\times X \to X$ is an action. That means that $G$ is a group (in this paper we will restrict ourselves to countable groups), the function $T$ is continuous, and satisfies $T(1_G,x)=x$ for all $x\in X$ (where $1_G$ is the identity of $G$) and $T(g,T(g',x))=T(gg',x)$ for all $g,g'\in G$ and $x\in X$. We write $T(g,x)$ as $T^g(x)$, and the conditions stated above can be concisely rewritten as $T^{1_G}={\rm id}_X$ and $T^{g}\circ T^{g'}=T^{gg'}$ for all $g,g'\in G$. We adopt this notation throughout the whole paper. 
We say that $(X,T,G)$ is \define{faithful} (or simply that the action of $G$ is faithful) if $T^g\neq {\rm id}_X$ for all $g\neq 1_G$. The system $(X,T,G)$ is \define{free} (or the action of $G$ is free) if $T^g(x)\neq x$ for all $x\in X$ and all $g\neq 1_G$. 

The concept of non-determinism has a rich development in the theory. For integer actions, we refer to \cite{Hochman_determinism_top_dyn:2012} for interesting connections with entropy. We continue the study of non-determinism as started in \cite{Donoso_Maass_Petite_geometric_asymptotic:2024}, where a form of the following definition was proposed \footnote{In that paper, the notion of a horoball of non-determinism was introduced, but here we choose to take a functional approach and name the horofunctions of non-determinism instead of their level sets.}.  

\begin{definition}
    Let $(X,T,G)$ be a topological dynamical system and $\varepsilon >0$. We say that a horofunction $h \in  \partial G$ is $\varepsilon$-\define{deterministic} if $d(T^gx,T^gy)\leq\varepsilon$ for all $g\in\{h<0\}$ implies that $x=y$,  for all $x,y\in X$.
\end{definition}

Conversely, a horofunction that is not deterministic will be called non-deterministic.
\begin{definition}
    Let $(X,T,G)$ be a topological dynamical system and $\varepsilon >0$. A horofunction $h \in  \partial G$ is said to be $\varepsilon$-\define{non-deterministic} if there exist distinct points $x,y\in X$ such that $d(T^gx, T^gy)\leq \varepsilon$ for all $g\in \{h<0\}$. The pair $(x,y)$ is called an $(h,\varepsilon)$-\define{asymptotic pair}.
\end{definition}

We let $\ND_\varepsilon(X,T,G)$, or simply $\ND_\varepsilon(X)$ when there is no harm of confusion, denote the collection of all $\varepsilon$ non-deterministic horofunctions. That is,
$$\ND_\varepsilon (X)= \{ h \in \partial G \mid h \textrm{ is } \varepsilon\textrm{-non-deterministic} \}.$$
More generally, for a factor map $\pi\colon X\to Y$ between topological dynamical systems $(X,T,G)$ and $(Y,T,G)$, we let $\ND_{\varepsilon,\pi} (X)$ denote the set of $\ve$-non-deterministic horofunctions, where the $\ve$-asymptotic pair can be chosen in the same fiber. That is, $h\in \ND_{\ve,\pi}(X)$ if there exist distinct points $x,y\in X$ with $\pi(x)=\pi(y)$ and $d(T^gx, T^gy)\leq \varepsilon$ for all $g\in \{h<0\}$. We say $\ND_{\ve , \pi}(X)$ is a relative version of $\ND_\ve (X)$, and satisfies
\[
\ND_{\ve, \pi}(X) \subseteq \ND_{\ve}(X) \textrm{ and } \ND_{\ve}(X) \setminus \ND_{\ve, \pi}(X) \subseteq \ND_{\ve}(Y).
\]

The following is one of the main results of \cite{Donoso_Maass_Petite_geometric_asymptotic:2024}.
\begin{theorem}[{\cite[Theorem 4.3]{Donoso_Maass_Petite_geometric_asymptotic:2024}}]
\label{thm:robinson_crusoe}
    Let $(X,T,G)$ and $(Y,S,G)$ be two topological dynamical systems, where $G$ is an infinite group with a proper right-invariant metric, and let $\pi\colon X\to Y$ be a factor map. Then, either
    \begin{itemize}
        \item $\pi$ is bounded-to-1,
        \item or for any $\varepsilon>0$ there exists a horofunction $h \in \ND_{\ve,\pi}(X)$. 
    \end{itemize}
\end{theorem}
In particular, when $(Y,S,G)$ is the trivial one point system and $X$ is infinite, only the second point occurs and there must exist a $\varepsilon$-non deterministic horofunction. So any infinite system has at least one non-deterministic horofunction.  

As we shall see in \cref{sec:TopoHorofct}, we can show the more general \Cref{theo:factors_general_cone} which provides a directed version of \Cref{thm:robinson_crusoe}, in the sense that the horofunction considered is non-negative at a prescribed element. To prove it, we have to introduce the notion of {\it repulsion} of a set for a semi-group. 
For a set $S \subseteq G$, we denote by $\langle S \rangle_+$ the semi-group generated by the elements of $S$, i.e., the set of elements of the form $s_1\cdots s_n$, where each $s_i$ belongs to $S$.
 
\begin{definition} \label{def:repulsion} Let $(X,T,G)$ be a topological dynamical system, $S \subseteq G$ and $Y \subseteq X$ be a closed $T$-invariant subset. We say that $Y$ is pointwise $S$-repulsive if there exists $\delta>0$ such that for any finite set $F\subseteq X\setminus Y $, 
$$
\textrm{dist}(T^gF,Y) \le \delta \textrm{ for only finitely many } g \in \langle S \rangle_+.$$
\end{definition}
When $\langle S \rangle$ is infinite, this property  means that a finite set of points close to $Y$ can be separated from $Y$ under the iteration of $T^g$ for some $g\in \langle S \rangle_+$.

For instance, consider a $\Z$-action on the circle with a north-south dynamics; the (repelling) north fixed point is pointwise $\{ 1\}$- repulsive but not pointwise $\{ - 1\}$- repulsive. Note that in \cref{def:repulsion} restricting $F$ to being a singleton yields the same property. 
Actually, we will use the converse of this notion, that is,  when the set $Y$ is not $S$-repulsive.  Notice that non $S$-repulsiveness is meaningful when $\langle S \rangle$ is infinite.

With this notion of repulsion, we will use a general topological theorem, coined the Directed Robinson Crusoe theorem (\cite[Theorem 3.4]{Donoso_Maass_Petite_geometric_asymptotic:2024}), that we recall here. Below, for an unbounded subset $G'\subseteq G$, we use the notation $\partial G'$ to denote $\overline{b(G')} \setminus b(G')$, that is, the horofunctions that can be obtained along $G'$.

\begin{theorem}[Directed Robinson Crusoe theorem \cite{Donoso_Maass_Petite_geometric_asymptotic:2024}] \label{thm:directed_RC} 

Let $(X,T,G)$ be a topological dynamical system where $G$ is an infinite  group with a proper right invariant distance. 
Let $G_0 \subseteq G$ be an unbounded subset with an unbounded  complementary $G \setminus G_0$.

Assume that 
\begin{itemize}
    \item $O \subsetneq X$ is an open, not closed,  $G$-invariant  subset of $X$;
    \item the set  $\partial O$ is non pointwise $\tilde{S}$-repulsive for some  finite subset $\tilde{S} \subseteq \bigcap_{h \in {\partial G_0}} \{h<0\}$.
\end{itemize}
Then, for any neighborhood  $U$ of the border $\partial O$  there exists a horofunction $h\in \partial [G\setminus  G_0^{-1}]$ such that 
$$O \cap \bigcap_{g \in \{h<0\} }T^{g^{-1}}(U)\neq \emptyset.$$
\end{theorem}

For the applications of this theorem to the structure of the set of non-deterministic horofunctions of Heisenberg groups, we need to find horofunctions with additional constraints. We do this in the following subsection.

\subsubsection{Additional constraints on the directions for the directed Robinson Crusoe Theorem }

Let $G$ be a group with a right invariant and proper metric $\rho$. We say that a subgroup  $Z\leqslant Z(G)$ is {\em horostable} if $\{h<0\}Z=\{h<0\}$ for all horofunction  $h \in \partial G$. We say that the pair $(L,Z)$ is a {\it decomposition of} $G$ if $Z$ is horostable and $L$ is a left transversal of $Z$. That is, any $g\in G$ can be uniquely written as $g=az$ where $a\in L$ and $z\in Z$. We use the notation $g_L$ and $g_Z$ for the elements in $L$ and $Z$ such that $g=g_Lg_Z$. \\

To simplify notation, let us introduce the following distances. We define the distance to the center $N(g) = \inf_{z\in Z}\rho(zg,1_G)$, and then the distance $\dist_Z(g,f) = N(gf^{-1})$. For $R>0$ define a \define{tube} as $T_R(g)=ZB_{R}(g)$. Then, $T_R(f) = \{g\in G\mid \dist_Z(g,f)\leq R\}$.\\

Because $G$ is discrete and its metric $\rho$ is proper, there exists $\nu>0$ such that $B_{\nu}(1_G) = \{1_G\}$. This discreteness constant will be used for the results that follow. Finally, for a subset $A\subseteq L$ we consider the set $\partial A^{-1}$ of limits $b_{g_L}$ where $g\in A^{-1}Z$. If $A^{-1}\subseteq L$ this coincides with the limits along $A^{-1}$.\\

The decomposition is \define{asymptotically orthogonal} 
if for any sequence $g_n\in L$ with $g_n\to \infty$ we have  $\rho(g_n,1_G)-N(g_n)\to 0$.
This is the case, for instance, if we choose representatives of the cosets of $G/Z$ that asymptotically achieve the infimum distance to the origin within the coset. 

\begin{remark}
\label{rem:quant_asy_decomp}
    Consider a decomposition $(L,Z)$. Notice that this decomposition is asymptotically orthogonal if and only if for every $\delta>0$, there exists $R>0$ such that $0\leq \rho(g,1_G) - N(g)\leq \delta$ for every $g\in L$ such that $\rho(g,1_G)\geq R$.
\end{remark}

We present below two geometrical lemmas on the horofunctions arising from asymptotically orthogonal decompositions.  The first one states that a large enough tube $T_R$ contains locally a horoball.

\begin{lemma}\label{lem:geomHoroball2_tube}
Let $G$ be a discrete group with a right-invariant and proper metric $\rho$, let $(L,Z)$ be an asymptotically orthogonal decomposition of $G$ and $A \subseteq L$ an unbounded subset. Then, for any $M>0$ there exists an integer $n_{0} \in \N$ such that for any $\bar{g}\in G$, and $g \in ZA\bar{g}$ with $\dist_Z(g,\bar{g}) \ge n_{0}$ one can find a horofunction $h\in \partial A^{-1}$ such that
$$ \left[\overline{\{h<0\}} \cap \overline{T_{M}(1_G)}  \right] g \subseteq \{k\in G \mid \dist_Z(k,\bar{g})<\dist_Z(g,\bar{g})\}.$$   
\end{lemma}  

\begin{proof}
By contradiction, suppose that there exist constants $M>0$, and sequences $(\bar{g}_{n})_{n\in \N}$ and $(g_{n})_{n\in \N}$ such that $g_n\in ZA\bar{g}_n$, $\dist_Z(g_n,\bar{g}_n)$ goes to infinity, and such that for any horofunction $h \in \partial A^{-1}$ we have that
\begin{align}\label{eq:nonEmpty_tube}
\left[ \overline{\{h<0\}} \cap \overline{T_{M}(1_G)}  \right]g_n  \not\subseteq  \{k\in G \mid \dist_Z(k,\bar{g}_n)<\dist_Z(g_n,\bar{g}_n)\}. 
\end{align} 
Define the elements $w_n = g_n\bar{g}_n^{-1}\in ZA$, and $\ell_n = (w_n^{-1})_{L}\in L$. Then, $w_n^{-1} = \ell_nz_n$ with $z_n\in Z$. By the centrality of $z_n$, we obtain $w_n^{-1}f^{-1} = \ell_n f^{-1} z_n$ for any $f\in G$. Now, because the distance to $Z$ is constant along $Z$-cosets, we have 
\[N(fw_n) = \dist(Z,\ell_nf^{-1})\leq \rho(\ell_n,f), \]
and $N(w_n) = N(\ell_n)$. Define $\delta_n = \rho(\ell_n,1_G)-N(\ell_n)$. Then,
\begin{align}\label{eq:desigualdad_horo}
    N(fw_n) - N(w_n) \leq \rho(\ell_n, f) - \rho(\ell_n,1_G) + \delta_n = b_{\ell_n}(f) + \delta_n.
\end{align}
We may assume that  $(b_{\ell_n})_{n\in \N}$ converges to a horofunction $h \in \partial A^{-1}$. 

Now, by \eqref{eq:nonEmpty_tube}, for each $n\in\N$ there exists $f_n\in\{h<0\}\cap T_M(1_G)$ such that $\dist_Z(f_ng_n, \bar{g}_n)\geq\dist_Z(g_n,\bar{g}_n)$. In particular, this implies $N(f_nw_n)\geq N(w_n)$. By horostability, the horoball $\{h<0\}$ can be expressed as a union of $Z$-cosets. Because $T_M(1_G) = ZB_M(1_G)$ where $B_M(1_G)$ is finite, and $N$ is invariant under $Z$, up to taking a subsequence we suppose $f_n = f$ such that $h(f)<0$.

Consider $\delta =-\frac13h(f)$. By~\cref{rem:quant_asy_decomp}, there exists $R >0$ such that
\[\delta_n =  \rho(\ell_n,1_G)-N(\ell_n) \leq \delta,\]
when $N(w_n) = N(\ell_n)\geq R$. Because $N(w_n) = \dist_Z(g_n,\bar{g}_n)\to\infty$, the inequality holds for sufficiently big $n$. Since $b_{\ell_n}(f)\to h(f)$, for sufficiently big $n$, we have $b_{\ell_n}(f)\leq h(f) + \delta$. Finally, and once again for sufficiently big $n$, \eqref{eq:desigualdad_horo} holds for $f$ and we obtain
\[0 \leq N(fw_n) - N(w_n) \leq h(f) + 2\delta = \frac{1}{3}h(f) < 0,\]
which is a contradiction.
\end{proof}

The next lemma roughly states that some translation of a "cone" $G_0$ has a larger intersection with a tube $T_R$ than an untranslated $G_0$.
\begin{lemma}\label{lem:geomHoroball3_tube}
Let $G$ be a discrete group with a right-invariant and proper metric $\rho$, $(L,Z)$ an asymptotically orthogonal decomposition of $G$, let $G_0 \subseteq L$ be an unbounded subset and let $\eta >0$. Assume that $g \in G$ is such that $h(g^{-1})<-\eta $ for all $h \in \partial G_0$. Then, there exists an integer $n_{1}\in \N$ such that for any real number $r\ge n_1$
\begin{equation} \label{eq:tubecontainment}
\left[ ZG_0 \cap T_{r+\eta}(1_G)\right] g  \subseteq  T_{r}(1_G).\end{equation}
\end{lemma}
\begin{proof}
Take $g\in G$ with $h(g^{-1})<-\eta$ for all $h\in \partial G_0$. If the result does not hold, there exists a sequence $(g_n)_{n\in \N}$ in $ZG_0$ and a sequence $(r_n)_{n\in \N}$ of positive real numbers going to infinity such that $g_n \in T_{r_n+\eta}(1_G)$ but $g_n\notin T_{r_n}(g^{-1})$ for all $n\in \N$. As the inclusion is invariant under multiplying by $Z$, we may assume $g_n\in G_0\cap T_{r_n+\eta}(1_G)\setminus T_{r_n}(g^{-1})$. In particular, $g_n\in L$.
Up to a subsequence, we can assume that $b_{g_n}$ converges to some $h\in\partial G_0$. Fix $\delta>0$. Since $Z$ is central, and $\rho$ is right invariant, we have $\inf_{z\in Z}\rho(g_n,z) = N(g_n)$. Then, $g_n\in T_{r_n + \eta}(1_G)$ is equivalent to $N(g_n) \leq r_n+\eta$. As $g_ng\notin  T_{r_n}(1_G)$, we have in particular that $\rho( g_ng,1_G)\geq r_n$. Furthermore, $\rho(g_n,1_G)\geq \rho(g_ng,1_G) - \rho(g,1_G)>r_n - \rho(g,1_G)$. So, by~\Cref{rem:quant_asy_decomp}, for sufficiently large $n$, $\rho(g_n,1_G)-N(g_n)\leq\delta$. Therefore,
\[\rho(g_n,1_G) = N(g_n) + (\rho(g_n,1_G) - N(g_n))\leq r_n + \eta +\delta.\]
Finally, by right invariance $\rho(g_n,g^{-1}) = \rho(g_ng,1_G)$ so
\[b_{g_n}(g^{-1}) = \rho(g_n,g^{-1})-\rho(g_n,1_G)\geq r_n - (r_n+\eta+\delta) = -\eta-\delta.\]
Therefore, $h(g^{-1})\geq -\eta -\delta$. Because $\delta$ was arbitrary, we have $h(g^{-1})\geq-\eta$. This is a contradiction.    
\end{proof}

\begin{theorem}[Directed subspace Robinson Crusoe theorem] \label{thm:directed_RC_heisenberg} Let $(X,T,G)$ be a topological dynamical system, where $G$ is a discrete group with a right-invariant and proper metric $\rho$, and let $(L,Z)$ be an asymptotically orthogonal decomposition of $G$. Consider $G_0 \subseteq L$ an unbounded subset where $L \setminus G_0$ is unbounded.

Assume that 
\begin{itemize}
\item $F \subseteq X$ is a closed, non-isolated, $G$-invariant subset of $X$,
\item There exists a finite subset $\tilde{S} \subseteq \cap_{ h \in {\partial G_0}} \{h<0\}$ such that for any $\delta>0$ there exists $x\notin F$ and infinitely many $g\in \langle \tilde{S}\rangle_+$ with $d(T^{cg}x,F)\leq \delta$ for all $c\in Z$.   
\end{itemize}
Then, for any $\varepsilon>0$, there exists a horofunction $h  \in \partial [L\setminus  G_0]^{-1}$ and $x\notin F$ such that 
\[ d(T^gx,F)< \varepsilon \text{ for all } g\in \{h<0\}.\]

\end{theorem} 

\begin{proof} 
For the sake of contradiction, assume that for any horofunction $h \in \partial [L\setminus  G_0]^{-1}$ and $x\notin F$ there exists $g \in \{h<0\}$ with  $d(T^gx,F)\geq \varepsilon$. Set $\ve_1 = \ve/2$. Using compactness arguments, there exists $M>0$ such that for any $x\in X$, with $d(x,F)\geq \varepsilon_1$ and any $h \in \partial [L\setminus  G_0]^{-1}$, there exists some $g\in \{h<0\}\cap B_{M}(1_G)$ such that for any $a \in B_\nu (1_G)$, $d(T^{ag}x,F)> \varepsilon_1$.   
If there was some $h\in\partial[L\setminus G_0]^{-1}$ such that $\{h<0\}= \emptyset$, the conclusion of the theorem would hold by vacuity. Therefore, we assume that the horoball of every $h\in\partial[L\setminus G_0]^{-1}$ is non-empty, and we may assume $B_M(1_G)$ intersects every horoball. 

Let $n_0$  be the integer given by \cref{lem:geomHoroball2_tube} associated with the constants $M+1$ and the set $A= L \setminus G_0$. Let $S = \tilde{S}^{-1}$. Then, by assumption $S$ is a finite subset of $\bigcap_{h \in \partial G_0}\{g:  h(g^{-1}) <0 \}$. Since $\partial G_0$ is compact, and $\tilde{S}$ is finite, we can define
\[\eta = \frac12\min_{s\in\tilde{S}}(-\max_{h\in\partial G_0}h(s)) > 0.\]
Then, $h(s^{-1})<-\eta$ for all $s\in S$ and $h\in\partial G_0$. Thanks to \cref{lem:geomHoroball3_tube}, there exists $n_1$ such that \eqref{eq:tubecontainment} holds for each $g\in S$. We also may assume $\overline{T_{n_0}(s)} \subseteq  T_{n_1}(1_G)$ for any $s \in S$.

As $F$ is invariant, we can take $0<\delta\leq \ve_1$ such that $d(x,F)< \delta$ implies $d(T^gx,F)\leq\varepsilon_1$ for all $g\in B_{n_1}(1_G)$. By assumption, for this $\delta$ there exists $x\notin F$ such that $d(T^{cg}x,F)< \delta$ for infinitely many $g\in \langle \tilde{S}\rangle_+ $ and all $c\in Z$. This implies that $d(T^{ckg}x,F)\leq \varepsilon_1$  
for infinitely many $g\in \langle \tilde{S}\rangle_+ $, for all $k\in  B_{n_1}(1_G)$ and all $c\in Z$. 
That is 
\begin{equation} \label{eq:initial_tube}
    d(T^{k}x,F)\leq \varepsilon_1  \text{ for all  } k\in T_{n_1}(g)
\end{equation}
for infinitely many $g\in \langle \tilde{S}\rangle_+ $. 

Define
\[
C(x)  =\{g \in G: d(T^{cg}(x),F) \leq \varepsilon_1 \quad \forall c \in Z\}.
\]
By the choice of $\delta$, $C(x)$ contains $T_{n_1}(g)$ for infinitely many $g\in \langle \tilde{S}\rangle_+ $. 
First, we claim it satisfies the following property:
\begin{equation}\label{eq:propStable} 
 \tag{P}
  \parbox{\dimexpr\linewidth-7em}{
If for  $g \in G$, there exists $h \in \partial[ L \setminus G_0]^{-1}$ such that \[\left[\{h<0\}\cap T_{M}(1_G)\right]g \subseteq B_\nu(1_G)C(x),\]  then $g \in C(x)$}
\end{equation}
Indeed, for every $k \in \{h<0\}\cap T_{M} (1_G)$ (nonempty by the definition of $M$) there exists $a \in B_\nu(1_G)$ such that $d(T^{akgc}x,F) \leq \varepsilon_1 $, $\forall c \in Z$.  The definition of $M$ implies that $d(T^{gc}(x),F) \leq \varepsilon_1$ for all $c\in Z$. That is $g \in C(x)$. 

Assume that $C(x)$ contains $T_{R} (\bar{g})$ of radius $R \ge n_1$ for some $\bar g\in C(x) $. We claim that for any $s \in S$ the  set $C(x)$ also contains $T_{R+\eta}(s\bar{g})$. 
Otherwise, there is $k^* \in T_{R+\eta}(s\bar{g}) \setminus C(x)$ minimizing  the distance $\dist_Z(k^*, s\bar{g})$. We may assume that $k^*(s\bar{g})^{-1}$ lies in $L$ by the $Z$ invariance of the tube and the set $C(x)$.
The minimality condition implies that every $k$ with $\dist_Z(k,s\bar{g}) < \dist_Z(k^*,s\bar{g})$ lies in $C(x)$. Also, notice that by the choice of the constant $n_1$,  $\overline{T_{n_0}(s\bar{g})}$ is included in $T_{R}(\bar{g})$, hence in $C(x)$. So we have  $\dist_Z(k^*, s\bar{g}) > n_0$.  
If  $k^*$ belongs to $[L\setminus G_0]s\bar{g}$, \cref{lem:geomHoroball2_tube} ensures the existence of a horofunction $h \in \partial [L\setminus G_0]^{-1}$ such that 
$$\left[\{h<0\} \cap T_{M+1}(1_G)\right] k^*  \subseteq\{k\in G \mid \dist_Z(k,s\bar{g})<\dist_Z(k^*,s\bar{g})\} \subseteq  C(x)\subseteq B_\nu(1_G)C(x).$$ 

This is impossible by Property \eqref{eq:propStable}.
It follows that $k^*$ belongs to $ G_0s\bar{g}$. Then, \cref{lem:geomHoroball3_tube} applied to $g=s$,  implies that $k^*$ belongs to $T_{R}(1_G)\bar{g}\subseteq C(x)$: indeed, if $f = k^*(s\bar{g})^{-1}\in G_0$, we have that $N(f) = \dist_Z(k^*, s\bar{g})\leq R+\eta$. Then, $f\in ZG_0\cap T_{R+\eta}(1_G)$ and $fs\in T_R(1_G)$. Thus, $k^*= (fs)\bar{g}$. Again a contradiction. 

Iterating the argument, we get for any $s=s_n\cdots s_1$ in $\langle S\rangle_+$ that $C(x)$ contains the set
$\bigcup_{n\ge 1} \bigcup_{s_1, \ldots s_n \in S} T_{n_1+n\eta}(s_n\cdots s_1 g)$ for any $g$ such that $T_{n_1}(g)\subseteq C(x)$. Taking $g=\tilde s_1\cdots\tilde s_n$ among the infinitely many elements of $\langle\tilde S\rangle_+$ and choosing $s_j=\tilde s_{n-j+1}^{-1}$ such that $s_n\cdots s_1g=1_G$, we get $T_{n_1+n\eta}(1_G)\subseteq C(x)$. Indeed, only finitely many elements of $\langle\tilde S\rangle_+$ are products of at most $m$ generators, and $n$ may be taken arbitrarily large. As this happens for infinitely many $g\in \langle \tilde{S}\rangle_+ $, we get that $C(x)$ contains $T_{l}(1_G)$ for arbitrarily large $l$. Hence $C(x)=G$, that is, $d(T^gx,F)\leq\ve_1<\ve$ for every $g\in G$. This is a contradiction. 
\end{proof}

\section{Generalities and topological properties of the set of non-deterministic horofunctions}\label{sec:TopoHorofct}
In this section, we provide geometrical conditions on the metric of the group $G$ to get restrictions on the set of non-deterministic horofunctions of any topological dynamical system $(X,T,G)$. In \cref{sec:closednessND} we provide sufficient conditions to ensure the closedness of the set of such horofunctions (\cref{thm:closeness-ND}). These conditions apply in particular to nilpotent group actions. Then the  next section \ref{sec:IntersectionsND}   sets restrictions on the non-deterministic horoballs for an action: their intersection should only contain torsion elements (\cref{thm:vacia_tf}.)  The last subsection concerns non abelian groups. We show that the set of non-deterministic horofunctions is invariant under the action of $G$ on its border. We deduce a sufficient condition so that for any topological dynamical system $(X,T,G)$,  the intersection of its  non-deterministic horoballs is trivial (\cref{thm:Vacia_ConditionG}). 

\subsection{On the closedness of $\ND_\varepsilon(X)$}\label{sec:closednessND}
An important property of the half-spaces of non-determinism for $\Z^d$ action, in \cite{Boyle_Lind_expansive_subdynamics:1997} is that they form a nonempty closed subset (when endowed with the Grasmannian topology). 
For a general group action, we have only stated in \cref{thm:robinson_crusoe} that the set of non-deterministic horofunctions is always nonempty. However, examples in \cite{Donoso_Maass_Petite_geometric_asymptotic:2024} show that closedness does not always hold. In fact, this property fails even in $\Z^2$, using the $\ell^1$ norm. So, it is natural to ask under which conditions on the metric, the set $\ND_\varepsilon (X)$ is closed for all $\varepsilon>0$ (see also \cite[Question 5.10]{Donoso_Maass_Petite_geometric_asymptotic:2024}).  In this section, we find conditions on the distance on $G$ that guarantee that the set $\ND_\varepsilon(X)$ is a closed set in $\partial G$ for all $\varepsilon>0$. In such a case, since $\ND_\varepsilon (X) \subseteq \ND_{\varepsilon'} (X)$ when $\varepsilon \leq \varepsilon'$, the nested sequence of compact sets 
\[\ND(X)=\bigcap_{\varepsilon>0 } \ND_\varepsilon (X),\]
is non-empty. This implies that the horofunction provided by~\Cref{thm:robinson_crusoe} can be chosen uniformly in $\varepsilon$. The same conclusion can be drawn considering a factor map $\pi \colon X \to Y$ between $G$-systems, where we define \[\ND_{\pi}(X)=\bigcap_{\varepsilon>0 } \ND_{\varepsilon,\pi}(X).\] 
This is summarized in \Cref{thm:closeness-ND}.
We also deduce a geometrical restriction on the non-deterministic horofunctions (\Cref{thm:vacia_tf}). 

The key geometrical condition that ensures that $\ND_{\epsilon}(X)$ is closed for any system $(X,T, G)$ is captured by the following properties (see \cref{thm:closeness-ND}). 

\begin{definition}\label{def:qsubadditive}
Let $G$ be a group with a right invariant and proper metric $\rho$, and consider a closed subset $V\subseteq \partial (G, \rho)$.
\begin{enumerate}[label=C\arabic*),ref=C\arabic*)]
   \item\label{itemhorofct:1} We say that  $V$ 
   has \define{uniformly quasi-subadditive  horofunctions} if there exists $R>0$ such that \[h(gg') \le h(g) + h(g') +R\] for all $h\in V \cap \partial G$ and all $g,g' \in G$. 

    \item\label{itemhorofct:2} The set $V$ has no lower bound for its horofunctions (see Section \ref{sec:NoLowerBound}).
  
\end{enumerate} 

\end{definition}
{When  the set $V$ is the whole space $\partial G$, we say that the group $G$ has \define{uniformly quasi-subadditive horofunctions} and has \define{no lower bound for its horofunctions} when \ref{itemhorofct:1} and \ref{itemhorofct:2} occur respectively.}  

\begin{theorem} \label{thm:closeness-ND}
Let $G$ be a group with a proper right-invariant metric {and let $V \subseteq \partial (G, \rho)$  be a closed set with uniformly quasi subadditive horofunctions and no lower bound for its horofunctions.} Let $\pi\colon X\to Y$ a factor map between the topological dynamical systems $(X,T,G)$ and $(Y,T,G)$ that is not bounded-to-one. Then for any $\ve >0$, the sets $\ND_{\ve,\pi} (X)\cap V$ and  $\ND_{\pi}(X)\cap V$ are closed. 
\end{theorem}
{This theorem is meaningful when  $\ND_{\ve,\pi} (X)\cap V$ is not empty for any small enough $\ve$, for instance when $V= \partial G$ (\cref{thm:robinson_crusoe}). This implies in particular that $\ND_{\pi}(X)$ is non-empty. The localized version of this result (that is, when $V$ is a strict subset of the border) will be applied for the Heisenberg group (see Section~\ref{sec:horo-heisenberg}).}   

Note that taking $Y$ to be the trivial system, 
\cref{thm:robinson_crusoe} provides that the set of non-deterministic horofunctions $\ND_\ve (X)$ is not empty for any $\ve >0$ and infinite topological dynamical system $(X, T, G)$. 
If, in addition the horofunctions on $G$ are quasi subbaditive and have no lower bound, \cref{thm:closeness-ND} gives that the sets $\ND_\ve (X)$ and $\ND(X)$ are non-empty closed sets.  This shows \cref{thm:A}.\\

Before proving \cref{thm:closeness-ND}, let us make some comments on the hypothesis  \ref{itemhorofct:1} and \ref{itemhorofct:2}.  The group $\Z^{d}$ equipped with the Euclidean metric satisfies both conditions \ref{itemhorofct:1}-\ref{itemhorofct:2} for $V = \partial \Z^d$, since its horofunctions are linear. Thus, \Cref{thm:closeness-ND} serves as a natural generalization of the closedness of $\ND(X)$ for $\mathbb{Z}^d$ with the $\ell^2$ norm.

\begin{proof}[Proof of \cref{thm:closeness-ND}]
Let $\ve>0$. We may assume that for any pair of distinct points $x,y\in X$ with $\pi(x)=\pi(y)$ there exists $g\in \cap_{h\in V} \{h<0 \}$ such that $d(T^gx,T^gy)>\ve$, since otherwise $\ND_{\ve,\pi}(X)\cap V=V$ and there is nothing to prove. 

\begin{claim}  \label{claim:delta-sep}
There exists $\delta>0$ such that for any $h\in V \cap \ND_{\ve,\pi}(X)$, there exists $x,y\in X$ with $d(x,y)\geq \delta$, $\pi(x)=\pi(y)$ and $d(T^{g}x,T^{g}y)\leq \ve$ for all $g\in \{h<0\}$. 
\end{claim}
\begin{proof}[Proof of claim]
Let $h\in \ND_{\ve,\pi}(X)\cap V$ and let $x\neq y$ such that $\pi(x)=\pi(y)$ and $d(T^gx,T^gy)\leq \ve$ for all $g\in \{h<0\}$. Note that the set $F \coloneqq \{g \in G \mid d(T^{g}x,T^{g}y)>\ve\}$ satisfies that $F\neq \emptyset$ and $F\subseteq \{h\geq 0\}$. Let $R$ be the (uniform) constant of uniform quasi subadditivity for the horofunctions. Since $V$ has no lower bound for the horofunctions, there exist $L >0$ (independent of $h$) and $f \in B_L (1_G)$ such that $h(f) \le -2R-1$ (\cref{lem:UnifLowerBound}). Let $g' \in F$ such that $h(g') \le \inf h(F) +1$. It follows that for any $g$ in $\{h<0\}$,  \[h(gfg') \le h(g) + h(f) + h(g') + 2R < \inf h(F).\] In particular, this implies that $\{h<0\}fg'$ and $F$ are disjoint.

Since $B_L(1_G)$ is compact, the family of transformations $\{T^g \}_{g\in B_L(1_G)}$ is equicontinuous and we may take $0<\delta<\ve$ such that for any  $x', y' \in X$, $d(x',y')>\ve$ implies $d(T^{g}x',T^{g}y')>\delta$ for all $g\in B_L(1_G)$. Notice that since $d(T^{g'}x, T^{g'}y) > \ve$, we have $ d(T^{fg'} x, T^{fg'}y)> \delta $ 
and because $\left(\{h <0\}fg' \right)\cap F= \emptyset$, we have  $ d(T^{gfg'} x, T^{gfg'}y) \le  \epsilon $ for any $g \in \{ h<0 \}$. The pair of points $T^{fg'} x$ and $T^{fg'} y$ verify the claim.\end{proof}
By \Cref{claim:delta-sep}, it is straightforward to verify that $V\cap \ND_{\ve,\pi}(X)$ is closed. The closedness of $V\cap\ND_{\pi}(X)$ follows immediately. 
\end{proof}

\begin{remark}
From the proof of \Cref{thm:closeness-ND}, observe that when any horofunction of $(G,\rho)$ is subadditive, that is, for every $g,g'\in G$ and $h\in \partial G$, $h(gg')\leq h(g) + h(g')$, there is no need for $G$ to have no lower bound for the horofunctions. Also, given the constant $R$ in Condition \ref{itemhorofct:1}, condition \ref{itemhorofct:2} can be relaxed to $\sup_{h \in \partial G} \inf(h(B_L(1_G)))<-2R$, by setting, in the proof of  \cref{thm:closeness-ND}, $\gamma=-2R-\sup_{h \in \partial G} \inf(h(B_L(1_G)))>0$ and choosing $g'\in F$ with $h(g')\leq \inf h(F) + \gamma$.

\end{remark}

\subsection{The intersection of non-deterministic horoballs}\label{sec:IntersectionsND}
In this section, we exhibit restrictions on the horoballs defined by non-deterministic horofunctions. We address Question 5.8 raised in \cite{Donoso_Maass_Petite_geometric_asymptotic:2024} and provide conditions that yield a positive answer. We provide a geometrical condition ensuring that for any topological dynamical system $(X,T,G)$, the intersection of its non-deterministic horoballs contains only torsion elements (\cref{thm:vacia_tf}). In particular, it is trivial if $G$ is torsion-free.

\begin{definition}\label{def:SymHoroball}
    A group $G$ with a right invariant and proper metric $\rho$ has a \define{symmetric horoballs space} if for any sequence $(g_n)_{n\in\N}$ such that both the limits $h = \lim_{n} h_{g_n}$ and $h'= \lim_{n} h_{g_{n}^{-1}}$ exist, the set $\{g \in G \mid  h(g^{-1})<0\}$ coincides with $\{g \in G \mid h'(g)<0\}$.
\end{definition}
 
 It is simple to check that the horoballs space of the abelian group $\Z^d$ is symmetric for the Euclidean metric, but fails to be symmetric for the $\ell^1$-metric. 

\begin{definition}
    Let $k\in G$. The \define{half-horoboundary associated with $k$} is the set of horofunctions $h\in\partial G$ such that $h(k)\geq 0$. We denote this set by $\Cone(k)$.
\end{definition}

For instance, if $G= \Z^2$, the set of $\ell^2$ horofunctions in $\Cone(k)$ is a half-circle of directions (here we identify a horofunction $h$ with the unit vector $v_h$ such that $h(x)=\langle x, v_h\rangle$). 
\begin{center}
\begin{figure}[H] 
\begin{tikzpicture}[>=latex, scale=1]

    \def\R{1} 
    \def\kangle{45} 
    \def\kdist{1.5} 

    \draw[very thin, ->, gray] (-\R*1.3, 0) -- (\R*1.3, 0);
    \draw[very thin, ->, gray] (0, -\R*1.3) -- (0, \R*1.3);

    \draw[black] (0,0) circle (\R);

    \fill[black] (\kangle:\kdist) circle (1.5pt) node[anchor=south west, inner sep=2pt] {$k$};
    
    \draw[dashed, thin] (0,0) -- (\kangle:\kdist);

    \draw[thin, black] ({\kangle+90}:\R) -- ({\kangle-90}:\R);
    
    \draw[blue, thick] ({\kangle-90}:\R) arc ({\kangle-90}:{\kangle+90}:\R);

\end{tikzpicture}
\caption{Half-horoboundary associated with $k$, for the Euclidean metric. }
\end{figure}
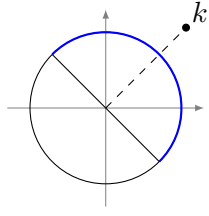 
\end{center} 

\begin{lemma} \label{lem:cone_nonempty}
  Let $G$ be an infinite discrete group with a right invariant and proper metric $\rho$.  Then for any $k \in G$, $\Cone(k)$ is a closed nonempty subset of $\partial G$.
\end{lemma}

\begin{proof}
 The closedness of $\Cone(k)$ is clear. We show that it is nonempty. 
 First, assume that $k$ is torsion and let $m\in \N$ with $k^m=1_{G}$. Let $g_i$ be a sequence that goes to infinity. Because $\sum_{j=0}^{m-1} \rho(g_i,k^{j+1})-\rho(g_i,k^j)=0$, we have that for each $i$ there exists $j(i)\in\{0,\ldots,m-1\}$ with $\rho(g_i,k^{j+1})-\rho(g_i,k^j)\geq 0$. After taking a subsequence, we may assume $j(i)=j$ is constant and define $g'_i=g_ik^{-j}$. Then $\rho(g_i',k)-\rho(g_i',1_G)\geq 0$ and then any limit point of $b_{g_i'}$ belongs to $\Cone(k)$. 
If $k$ is not torsion, then consider $\varphi(n) = \rho(k^n,1_G)$, and $c = \rho(k,1_G) = \varphi(1)$. Notice that $|\varphi(n)- \varphi(n-1)|\leq c$. Next, the set $\{k^n\mid n\leq 0\}$ is unbounded by properness as $k$ is not torsion. Then, consider $i>0$ and $N\leq 0$ such that $\varphi(N) \geq i$. Define $m_i= \textnormal{argmax}\{\varphi(n) \mid n\in[N,0], \varphi(n)\geq i\}$. Because $\varphi(0) = 0$, we have that $m_i+1\in[N,0]$. Thus, $\varphi(m_i+1)\leq\varphi(m_i)$ and 
\[\varphi(m_i+1)\geq \varphi(m_i)-c\geq i-c.\]
Taking the limit $i\to\infty$, and defining the sequence $(n_i)_{i\in\N}$ by $n_i = m_i+1$, we have that $\rho(k^{n_i-1},1_G) \geq \rho(k^{n_i},1_G)\to\infty$. It is straightforward to check that any limit point of $b_{k^{n_i}}$ belongs to $\Cone(k)$.  
\end{proof}

The next theorem is a refinement of \cite[Theorem 4.3]{Donoso_Maass_Petite_geometric_asymptotic:2024}. It ensures the existence of a non-deterministic horofunction in $\Cone(k)$.\footnote{A version of this result already appeared in an early draft of \cite{Donoso_Maass_Petite_geometric_asymptotic:2024}, but was not included in the final publication because it had no application there. The proof is included here for completeness and follows the original one, with minor modifications.}

{
\begin{theorem}\label{theo:factors_general_cone}
	Let $(X,T,G)$ and  $(Y,S,G)$ be topological dynamical systems, where $G$ has a metric $\rho$ with symmetric horoballs space. Let $\pi \colon X \to Y$ be a factor map that is not bounded-to-one. 
	Let $k \in G$ be an element of infinite order. Assume that $\ND_{\ve,\pi}(X)$ is closed for some $\ve>0$. Then there exists a horofunction $h \in \Cone(k) \cap \ND_{\ve,\pi}(X)$. 
\end{theorem}

In particular, under the assumptions of \cref{theo:factors_general_cone}, if $X$ is infinite, and taking $Y$ to be the trivial system, we obtain that $\Cone(k)\cap \ND_{\varepsilon}(X)\neq \emptyset$ for all $\varepsilon>0$. 

\begin{proof}
The proof follows a similar strategy to that of \cite[Theorems 4.3 and 4.4]{Donoso_Maass_Petite_geometric_asymptotic:2024}. We will apply \cref{thm:directed_RC} to a product system.  We keep the same notations as in \cite{Donoso_Maass_Petite_geometric_asymptotic:2024}: i.e, 
we consider the system $(R_{\pi}, T^{(2)},G)$, where $R_\pi=\{(x,y)\in X\times X: \pi(x)=\pi(y)\}$ and $T^{(2)}$ is the diagonal action $T^{(2)}_g(x,y)= (T_g(x),T_g(y))$.
The set  $O$ denotes $R_\pi\setminus\triangle_X$, where $\triangle_X$ is the diagonal on $X$.
The set $O$ is clearly open and $G$-invariant. It is also not closed. To see this, note that since $\pi$ is not bounded to one, then there exists $x\in X$ and a sequence of distinct points $x_n$ with $(x,x_n)\in R_{\pi}$ for all $n$. That means that  $\triangle_X$ cannot be open in $R_{\pi}$, hence $O$ is not closed. Moreover, by a compactness argument, for any $\delta > 0$ there exists $\varepsilon' >0$ such that if $(x,y)\in O$ and ${\rm dist}(x,y)\leq \varepsilon'$ then ${\rm dist}((x,y),\partial O)\leq \delta$. 

We distinguish 2 cases.

\textbf{Case 1.} For any $h\in \partial G$, $h(k^{-1})\geq 0$. By the symmetric horoballs property, $h(k)\geq 0$ for all $h\in \partial G$ as well.  In this case,  \cref{thm:robinson_crusoe} is enough for our purposes. 

\textbf{Case 2.} There exists $\bar{h}\in \partial G$ such that $\bar{h}(k^{-1})<0$. Define, for $0<\eta<-\bar{h}(k^{-1})$, the set $G_\eta=\{ g\in G:  k^{-1}\in B_{\rho(1_G,g)-\eta}(g) \}$ and let $g_n$ such that $b_{g_n}\to \bar{h}$. Note that $G_\eta$ is unbounded since it contains all but finitely many $g_n$'s, and by definition $k^{-1}\in \bigcap_{h\in \partial G_{\eta}} \{h\leq -\eta\} \subseteq \bigcap_{h\in \partial G_{\eta}} \{h<0\}$. Also note that $G\setminus G_{\eta}$ is unbounded, otherwise, we would have $\partial G=\partial G_{\eta}$, so $h(k^{-1}) < 0$ for every $h\in\partial G$, and by \cref{lem:cone_nonempty} applied to $k^{-1}$ we would obtain $h\in\Cone(k^{-1})$, contradicting that $h(k^{-1}) < 0 $.  

We claim that  $\partial O$ is not pointwise $\{k^{-1}\}$-repulsive. 
Recall that for $\delta>0$, $\varepsilon'  >0$ is such that $(x,y)\in O$ and $d(x,y)\leq \varepsilon'$ implies $d((x,y),\partial O)\leq \delta$.  Let $B_1,\ldots, B_m$ be distinct open balls of diameter $\varepsilon'/2$ whose union covers $X$. 
Let $x_1,\ldots,x_{m+1}$ be distinct points with $\pi(x_1)=\pi(x_2)=\cdots=\pi(x_{m+1})$. Let $F=\{ (x_i,x_j): i\neq j \} \subseteq O$. Then, by the pigeonhole principle, for any $g\in \langle k^{-1} \rangle_+ \leqslant G$ there exist $i$ and $j$ such that $T^{g}(x_i)$ and $T^{g}(x_j)$ belong to a same ball $B_\ell$. This means that $d(T^g(x_i),T^g(x_j))\leq \varepsilon'$, which implies that $(T^g(x_i),T^g(x_j))$ is $\delta$-close to $\partial O$. The claim is proved. 
	
So the set $\partial O$ is non pointwise $\{k^{-1}\}$-repulsive. 
Applying the directed Robinson Crusoe Theorem \ref{thm:directed_RC} to $O$ and $G_0=G_{\eta}$, we obtain a horofunction in $\partial (G\setminus G_{\eta}^{-1})\cap \ND_{\ve,\pi}(X)$. By compactness, there exists $h\in \partial G$ and a sequence $\eta_j\to 0$ with $h_{\eta_j}\in \partial (G\setminus G_{\eta_j}^{-1})\cap \ND_{\ve,\pi}(X)$ such that $h_{\eta_j}\to  h$. By assumption $\ND_{\ve,\pi}(X)$ is closed and then $h\in \ND_{\ve,\pi}(X)$. Noting that the sets $(G\setminus G_{\eta}^{-1})$ decrease as $\eta\to 0$, we conclude that $h\in \bigcap_{\eta>0} \partial(G\setminus G_{\eta}^{-1}) \cap \ND_{\ve,\pi}(X)$. 
	
By a diagonal argument, we may write  $h=\lim_{j\to \infty} b_{g_j}$ for a sequence $(g_j)_{j\in \N}$ going to infinity and such that  $g_j\in G\setminus G_{\eta_j}^{-1}$. Let $h'=\lim b_{g_j^{-1}}$, which exists after taking a subsequence, if needed. Since $g_j^{-1} \notin G_{\eta_j}$ we have that for all $j\in \N$, $\rho(k^{-1},g_j^{-1})\geq \rho(1_G,g_j^{-1})-\eta_j$, which means that $h'(k^{-1})\geq 0$. 
By the symmetric horoballs space assumption, we obtain $h(k)\geq 0$. That is, $h\in \Cone(k)\cap \ND_{\ve,\pi}(X)$, as desired.
\end{proof}

\begin{theorem}
\label{thm:vacia_tf}
    Let $(X,T,G)$ and  $(Y,S,G)$ be topological dynamical systems, where $G$ has a metric $\rho$ with symmetric horoballs space. Let $\pi \colon X \to Y$ be a factor map that  is not bounded-to-one and such that $\ND_{\varepsilon, \pi}(X)$ is closed for every $\ve>0$. 
Then, it holds
    \[\bigcap_{h\in\ND_\pi(X)} \{h<0\} \subseteq {\rm Tor}(G).\]
In particular, if the group $G$ is torsion-free, the above intersection is empty.
\end{theorem}
Under the assumptions of~\cref{thm:vacia_tf}, if $X$ is infinite, and taking $Y$ to be the trivial system, we obtain that $\ND(X)$ is not empty. This provides \cref{thm:B}. 
\begin{proof}
    Note that by definition, if  $k\in \bigcap_{h\in\ND_{\varepsilon,\pi}(X)}\{h<0\}$, and $k$ of infinite order,  then $\ND_{\varepsilon,\pi}(X)\subseteq\Cone(k)^c$, which contradicts $\Cone(k)\cap \ND_{\varepsilon,\pi}(X)\neq \emptyset$.
    If $\ND_{\varepsilon,\pi}(X)$ is closed for all $\varepsilon>0$ sufficiently small, and if $k$ is of infinite order, as $\Cone(k)$ is also closed, we get that the compact set $\Cone(k)\cap \ND_{\pi}(X)= \bigcap_{\varepsilon>0} \Cone(k)\cap \ND_{\varepsilon,\pi}(X)$ is nonempty. Repeating the above argument yields $\bigcap_{h\in\ND_{\pi}(X)}\{h<0\}\subseteq {\rm Tor}(G)$. 
\end{proof}

In what follows, we explore conditions so that the intersection in \cref{thm:vacia_tf} is actually empty. This will be the content of \cref{thm:Vacia_ConditionG}. 

\subsection{The action on the border and invariance of non-deterministic horofunctions}\label{sec:InvBorderAction}

In this section, we give a sufficient condition ensuring  the set of non-deterministic horofunctions to be invariant under the group action on its boundary (\cref{prop:condition_G}). This  phenomenon is particularly relevant for non-abelian groups. We then strengthen \cref{thm:vacia_tf}, by proving that for any topological dynamical system $(X,T,G)$, the intersection of its non-deterministic horoballs is empty (\cref{thm:Vacia_ConditionG}).

For   $G$  a countable group with a right-invariant and proper metric $\rho$, there is a natural topological dynamical system we can consider on the border: this is given by a left continuous action of $G$. That is, we may consider $(\partial G, C, G)$ where $C$\footnote{We choose to use the letter $C$ and not $T$ to distinguish the action from other topological dynamical systems. The letter $C$ stands for cocycle.} is given by 
\[C^g(h)(k) = h(kg) - h(g),\]
for all $k,g \in G$, and $h \in \partial G$. 

Note that $C^g(h)$ can also be described as follows: if $h = \lim_{n\to\infty} b_{g_n}$, then $C^g(h)=\lim_{n\to\infty}b_{g_ng^{-1}}$.

In \cref{prop:condition_G}, we provide a sufficient condition to ensure that the set of non-deterministic horofunctions is $C$-invariant for any topological dynamical system. That is, if $(X,T,G)$ is a topological dynamical system, $C^g(\ND_{\ve}(X))=\ND_{\ve}(X)$ for all $g\in G$ and all $\ve>0$. 

\begin{proposition}
\label{prop:condition_G}
    Let $(X,T,G)$ be a topological dynamical system with $(G,\rho)$ satisfying conditions \ref{itemhorofct:1} and \ref{itemhorofct:2}. Then, $C^g(\ND_{\varepsilon}(X)) = \ND_{\varepsilon}(X)$ for all $\varepsilon>0$ and all $g\in G$. 
\end{proposition}

\begin{proof}

We claim that the following property, about horoball shifts, holds:
\begin{equation}
 \tag{HS}\label{eq:horoball_shift}
  \parbox{\dimexpr\linewidth-4em}{
For every horofunction $h\in\partial G$ and every $t \in \R$, there exists $g\in G$ such that
    \[ \{h<0\} \subseteq \{h < t\}g.  \]
  }    
\end{equation}
Indeed, let $R$ given by \ref{itemhorofct:1}, and for $t\in \R$, using \ref{itemhorofct:2}, take $g\in G$ such that $h(g^{-1})<-R+t$. By \ref{itemhorofct:1} we have for $x\in \{h<0\}g^{-1}$, $h(x)< 0 +h(g^{-1})+R=t$. 

Let $\ve>0$. Note that by definition, for every $h\in\partial G$, and $g\in G$
    \[\{ C^g(h)<-h(g)\} = \{h<0\}g^{-1}.\]
 Now fix some $h\in \ND_{\ve}(X)$. As $G$ satisfies the property \eqref{eq:horoball_shift}, we may take $g'\in G$ such that
    \[\{C^g(h)<0\} \subseteq \{C^g(h)<-h(g)\}g'.\]
It follows that \begin{equation}\label{eq:good_inclusion}\{C^g(h)<0\}\subseteq \{h<0\}g^{-1}g'.\end{equation} As $h\in \ND_{\ve}(X)$, there exist $x\neq y$ such that $d(T_kx,T_ky)\leq \varepsilon$ for all $k\in \{h<0\}$. Setting $x'=T_{g'^{-1}g}x$ and $y'=T_{g'^{-1}g}y$, we have $x'\neq y'$, and using \eqref{eq:good_inclusion}, we get that $d(T_{k}x',T_{k}y')\leq \ve$ for all $k\in \{ C^g(h)<0 \}$. Thus, $C^g(h)\in \ND_{\ve}(X)$, as desired. 
\end{proof}

As we have seen, \cref{prop:condition_G} holds assuming only \eqref{eq:horoball_shift}. We do not currently know if this property has independent significance beyond abelian groups, or groups that already satisfy conditions \ref{itemhorofct:1} and \ref{itemhorofct:2}.

\begin{lemma} \label{lem:torsion_cocycle}
For any torsion element $k$, and any $h\in \partial G$, we have that $k\in \{C^g(h)\geq 0\}$ for some $g\in G$. 
\end{lemma}

\begin{proof} The proof follows a similar idea to that of \cref{lem:cone_nonempty}. Assume $k^m=1_{G}$. Let $g_i$ be a sequence that goes to infinity with $h=\lim b_{g_i}$. Because $\sum_{j=0}^{m-1} \rho(g_i,k^{j+1})-\rho(g_i,k^j)=0$, we have that for each $i$ there exists $j=j(i)\in\{0,\ldots,m-1\}$ with $\rho(g_i,k^{j+1})-\rho(g_i,k^j)\geq 0$. After taking a subsequence, we may assume $j(i)=j$ is constant and define $g'_i=g_ik^{-j}$. Then $\rho(g_i',k)-\rho(g_i',1_G)\geq 0$. It follows that $k\in \{C^{k^j}(h)\geq 0\}$.
\end{proof}

Joining the previous results together with Theorem~\ref{thm:vacia_tf}, we obtain the following.
\begin{theorem}\label{thm:Vacia_ConditionG}
    Let $G$ be a group that admits a metric $\rho$ with symmetric horoballs space, and suppose that 
    \begin{itemize}
         \item $G$ is torsion-free, or
        \item $(G,\rho)$ satisfies conditions \ref{itemhorofct:1} and \ref{itemhorofct:2}.
        \end{itemize}
Then, for any infinite topological dynamical system $(X,T,G)$ and any $\ve>0$ such that $\ND_{\ve}(X)$ is closed, 
    $$\bigcap_{h\in\ND_{\ve}(X)}\{h<0\} = \emptyset.$$
Furthermore, if $\ND_{\ve}(X)$ is closed for all sufficiently small $\ve>0$, then we can write $\ND(X)$ in place of $\ND_{\ve}(X)$ in the previous intersection.    
\end{theorem}
\begin{proof}
The conclusion when $G$ is torsion-free follows directly from \cref{thm:vacia_tf}. 
Now assume that $(G,\rho)$ satisfies conditions \ref{itemhorofct:1} and \ref{itemhorofct:2}. By \cref{thm:vacia_tf}, we only need to show that the intersection $\bigcap_{h\in\ND_{\ve}(X)}\{h<0\}$ does not contain torsion elements. For a torsion element $k$,  \cref{lem:torsion_cocycle} tells us that for any $h\in \ND_{\ve}(X)$, there exists $g\in G$ with $k\in \{C^g(h)\geq 0\}$. Thanks to \cref{prop:condition_G},  $C^g(h)\in \ND_{\ve}(X)$ and therefore $k\notin \bigcap_{h\in\ND_{\ve}(X)}\{h<0\}$.  
\end{proof}


\subsection{Expansiveness and distality}\label{sec:Distal}
In this section we comment on the relation between expansiveness and distality for group actions. We start recalling classical definitions. 
A pair $(x,y)$ is \define{proximal} if there exists $z\in X$ and a sequence $(g_n)_{n\in \N}$ in $G$ such that $(T^{g_n}x,T^{g_n}y)\to (z,z)$. We let $P$ denote the set of proximal pairs. A pair $(x,y)$ which is not proximal is said to be \define{distal}. 
A topological dynamical system $(X,T,G)$ is \define{distal} if $(x,y)$ is distal whenever $x\neq y$.

The following result, which is of independent interest, follows from \cite{Donoso_Maass_Petite_geometric_asymptotic:2024} using rather standard compactness arguments.
\begin{theorem} \label{thm:distal_expansive}
 Let $G$ be a finitely generated group. Let $(X,T,G)$ and $(Y,T,G)$ be topological dynamical systems, and let $\pi\colon X\to Y$ be a factor map which is not bounded-to-1. Assume that $(X,T,G)$ is expansive. Then there exist $x\neq y$ with $(x,y)\in P$ and $\pi(x)=\pi(y)$. 
 
 In particular, taking $Y$ to be the trivial system, we get that any expansive and distal topological dynamical $(X,T,G)$ system (with $G$ finitely generated) is finite.     
\end{theorem}

\begin{proof}
In view of \cref{lem:thick-horoballs-fg}, we may take a proper and right-invariant metric such that $(G,\rho)$ has thick horoballs. Let $c>0$ be an expansiveness constant of $(X,T,G)$. As $\pi$ is not bounded-to-1, by \cref{thm:robinson_crusoe} there exist $x\neq y$ with $\pi(x)=\pi(y)$ and $h\in \partial G$ such that $d(T^gx,T^gy)\leq c$ for all $g\in \{h<0\}$. Since $(G,\rho)$ has thick horoballs, we may take a sequence $g_n$ in $\{h<0\}$ such that $B_n(g_n)\subseteq \{h<0\}$. Taking a subsequence, we may assume that $T^{g_n}x\to x'$, $T^{g_n}y\to y'$. Then $d(T^gx',T^gy')\leq c$ for all $g\in G$. By expansiveness $x'=y'$, which means that $(x,y)\in P$.  

The impossibility of having infinite, expansive, and distal action for a group with thick horoballs (or equivalently, finitely generated) follows immediately. 
\end{proof}

Note that it is possible to have (infinitely generated) groups with a minimal, expansive and distal action on an infinite set. To the best of our knowledge, the first example of this type appeared in \cite[page 265]{McMahon_Wu_conectedness_homo_top_dyn:1976}. Although they did not prove it, it is not difficult to show that this example is indeed expansive. It is also worth noting that the sequence that makes the pair $(x,y)$ to be proximal lies inside the horoball $\{h<0\}$. 

\subsection{Expansiveness and subactions}\label{sec:ExpSub} 

In this section, we show that certain subactions of a given group action cannot be expansive. Our result relies on commutator computations. As we shall see, this applies in particular to the center of nilpotent groups, a result that will be needed for a different purpose in \cref{sec:limits_fold_tmsj}.

Let $G$ be a group, $K\leqslant G$ be a finitely generated abelian subgroup and denote 
\[Z_{K}=\{g\in G\mid  [g,k]\in Z(G) \text{ for all } k\in K\}.\] 
Then, standard computations show that $Z_K$ is a subgroup of $G$, and the function $\phi\colon K\times Z_K \to Z(G)$, $(k,g)\mapsto [k,g]$ is a bi-homomorphism. Set $A_{K}=\phi(K\times Z_K)$.
Note that the subgroup $\langle K,A_K\rangle$ generated by  $K$ and $A_K$ is abelian. To avoid trivial cases, we assume $K$ contains at least one element not belonging to the center of $G$, that $A_K$ is not trivial and that $G$ is torsion free.

The main result of this section is that for an action of the group $G$, the sub-action  of the subgroup $\langle K,A_K\rangle$ generated by  a commutative subgroup $K$ of $G$ with its associated $A_K$, always has a non-deterministic component along the directions given by the subgroup $K$. In particular, the $\langle A_K\rangle$-action  is never expansive.

To achieve this, let us introduce some notations. 
The abelian group  $\langle K,A_K\rangle$ will be denoted $H_K$.
Let $r_A \in \N$ denote the rank of any abelian group $A$, i.e. the $\Q$-dimension of the $\Q$-vectorial space $\Q \otimes A$. 

Since $K$ is finitely generated and $G$ is torsion free, $H_K=\langle K,A_K\rangle$ is a finitely generated torsion-free abelian group, hence free of finite rank $r_{H_K}$. We fix a basis of $H_K$, so that $H_K$ is isomorphic to $\mathbb Z^{r_{H_K}}$, and we endow $H_K$ with the inner product obtained by transporting the usual one on $\mathbb Z^{r_{H_K}}$ through this identification.
 
So the set of non-deterministic horofunctions for the subaction of $H_K$ (that are linear forms), may be identified with unit vectors in $\R^{r_{H_K}}$. 
Moreover we denote by  $\pi_A \colon \R^{r_{H_K}} \to \R^{r_{\langle A_K \rangle}}$  the orthogonal projection onto the $\mathbb R$-span of $A_K$ inside $\mathbb R^{r_{H_K}}$, so that $\ker\pi_A$ is its orthogonal complement.
For such an action, the next theorem guarantees the existence of such a non-deterministic horofunction in a restricted direction.

\begin{theorem} \label{thm:nonexpansive_subgroup}
 Let $(X,T,G)$ and $(Y,T,G)$ be topological dynamical systems, and let $\pi\colon X\to Y$ be a factor map which is not bounded-to-1. Assume that $G$ is torsion-free and let $K\leqslant G$ be a finitely generated abelian subgroup with non trivial $A_K$. 
 Then, for any non zero vector $\vec{w} \in \ker \pi_A \subseteq \mathbb{R}^{r_{H_K}}$, there is an element $\vec v \in 
 \ND_{\pi}(X,T, H_K) \cap \ker \pi_A $ such that $\langle \vec{w}, \vec{v}\rangle \le 0$.
\end{theorem}

\begin{proof}
The main idea is to start with a nondeterministic vector $v \in \R^{r_{H_K}}$, and using symmetric maps, obtain another nondeterministic vector which is closer to the half-space $\{ \langle \vec w, \cdot \rangle \leq 0 \}$. A limiting argument and the closedness of the set of nondeterministic vectors will allow us to obtain a nondeterministic vector in $\{ \langle \vec w, \cdot \rangle \le 0 \}$. 

From \cref{thm:closeness-ND} and \cref{thm:robinson_crusoe}, we know the set 
$\ND_{\pi}(X, T, H_K) \subseteq \R^{r_{H_K}}$ is a closed non-empty set.
If $\ND_{\pi}(X,T, H_K) \subseteq \ker \pi_A$, it suffices to find an element $\vec{v}\in\ND_{\pi}(X,T,H_{K})$ such that $\langle\vec{w},\vec{v}\rangle\leq0$. Suppose there is no such element, that is, $\langle\vec{w},\vec{v}\rangle>0$ for every $\vec{v}\in\ND_{\pi}(X,T,H_{K})$. Then, we have that $-\vec{w}\in\bigcap_{\vec{v}}\{x\in\R^{r_{H_K}}\mid \langle\vec{v},x\rangle < 0\}$. But, as $H_K$ is free abelian and endowed with the Euclidean metric, it has a symmetric horoball space. Thus, by~\cref{thm:vacia_tf},
\[\bigcap_{\vec{v}\in\ND_{\pi}(X,T,H_{K})}\{g\in H_K \mid \langle \vec{v},g\rangle <0\}=\emptyset.\]
Since $\ND_{\pi}(X,T,H_{K})$ is closed, we can extend the above identity to $\R^{r_{H_K}}$, leading to a contradiction.

We now prove the general case, that is, we assume there is a vector $\vec v_*\in \ND_{\pi}(X,T, H_K)$ that is not in $\ker \pi_A$.
Denote $H_{v}=\{x\in  H_K: \langle x,v\rangle <0\}$.

We will use the fact that the subaction of the group $H_K$ admits several symmetries. 
To see this, notice that the group $H_K$ is invariant under the conjugacy by any element $g\in Z_K$. Indeed this follows from the relations
\begin{align}
gag^{-1} &= a \quad \forall a \in A_K \quad (\leqslant Z(G)) \label{eq:nil1}\\
gkg^{-1} & = [g,k]k \in \langle K, A_K \rangle, \quad \forall k \in K.  \label{eq:nil2}
\end{align}
We denote then by $M_g$  the homomorphism of $ H_K$  $M_g \colon x \mapsto gxg^{-1}$. 
Moreover this homomorphism induces a (normalized)  map $M_g^*$ on the set of unit vectors in $\R^{r_{H_K}}$ by  $M_g^* v= \left(M_g^{-1}\right)^t v /\| \left(M_g^{-1}\right)^t v \|$, for all unit vector $v$, where  $M^t$ denotes the transpose (in the dual space) of the map $M$. In particular, we have that $M_g H_{\vec v}= H_{M_g^* \vec v}$.       

With these notations, the self-homeomorphism $T_g$ of $X$ satisfies the relation
\[ T^g \circ T^\ell = T^{M_g \ell} \circ T^g \quad \forall \ell \in \langle K, A_K \rangle.
\]
It is then simple to check that $M_g^* \vec{v} \in \ND_\pi(X, T, H_K)$ for any  $\vec v \in \ND_\pi(X, T, H_K)$. 

From the equation \eqref{eq:nil1}, notice that $M_g|_{\langle A_K\rangle } = {\rm Id}|_{\langle A_K\rangle }$ and from \eqref{eq:nil2} that $\left(M_g-{\rm Id}\right) (K) \subseteq A_K$. This implies that $(M_g-\rm Id_K)$ is nilpotent of degree 2 and $M_g$ is unipotent. Since $\vec v_*\notin\ker\pi_A$, there is $a\in A_K$ with $\langle\vec v_*,a\rangle\neq0$.
 Because $A_K=\phi(K\times Z_K)$, we write $a=[g,k]$ with $k\in K$ and $g\in Z_K$. By~\eqref{eq:nil2}, $(M_g-\mathrm{Id})(k)=[g,k]=a$, so $\langle\vec v_*,(M_g-\mathrm{Id})(k)\rangle=\langle\vec v_*,a\rangle\neq0$.
For this specific $g$, the matrix $M_g$ is not the identity. Then, standard computations about unipotent matrices show that there is a unit vector $\vec v_+$, that is fixed by $M_g^*$, such that $\{(M_g^*)^n\vec v_*\}$ converges to $\pm\vec v_+$ when $n$ goes
to $\pm\infty$\footnote{Recall that $\vec v_*$ is not in $\ker \pi_A$, which ensures the aforementioned convergences to $\vec{v}_+$ and $-\vec{v}_+$}. 
Moreover $\vec v_+$ belongs to $\ker \pi_A$. Since $\ND_\pi (X, T, H_K)$ is closed, $\vec v_+$ and $-\vec v_+$ both belong to this set. Hence one of the two has to belong to the half-space $\{ \langle w, \cdot \rangle \le 0\}$, showing the theorem. 
\end{proof}

We will apply   \cref{thm:nonexpansive_subgroup} in \cref{ex:heisenberg} to show the central direction of the Heisenberg group is not expansive (see \cref{cor=ZnonExp}).  

\section{Properties of non-deterministic horofunctions and horoballs in $\Z^d$, the discrete Heisenberg group and the free group} \label{sec:horo-heisenberg}

We describe the geometric properties of non-deterministic horofunctions across Euclidean,  nilpotent and hyperbolic geometries, providing examples of groups that satisfy the geometric conditions introduced earlier in Section \ref{sec:TopoHorofct}.

\subsection{Covering results in $\Z^d$}
By standard convexity arguments, we  exhibit a covering property of the horoballs of non-determinism for any $\Z^d$-systems. This will be useful in Section \ref{sec:LimitnFolding}.

\begin{lemma}
\label{lem:hahn_banach}
    Let $\{v_i\}_{i\in I}\subseteq \mathbb{S}^{d-1}$ be a closed set of directions such that $\bigcup_{i\in I}\{x\in\R^d \mid \langle v_i, x\rangle\leq 0\} = \R^d$. Then, $0$ belongs to $\Conv(\{v_i\}_{i\in I})$, the convex hull of $\{v_i\}_{i\in I}$ .
\end{lemma}

\begin{proof}
    Denote by $P$ the convex hull of the directions $V=\{v_i\}_{i\in I}$, and $E_i = \{x\in\R^d \mid \langle v_i, x\rangle\leq 0\}$. Because the set $V$ is closed, it is in particular compact and, therefore, $P$ is compact. Suppose $0\notin P$. Then, by the Hahn-Banach Theorem there exists $v\in\mathbb{S}^{d-1}$ and $a\in\R$ such that $\langle v,p\rangle < a < 0$ for all $p\in P$. By hypothesis, there exists an index $i_0\in I$ such that $-v\in E_{i_0}$, that is, $\langle v_{i_0}, v\rangle\geq 0$. This is a contradiction as $v_{i_0}\in P$. Therefore, the origin $0$ belongs to $P$.
\end{proof}

\begin{lemma}
\label{lem:Caratheodory}
    Let $V\subseteq \mathbb{S}^{d-1}$ be a closed set of directions such that $\bigcap_{v\in V}\{x\in\R^d \mid \langle v, x\rangle< 0\} = \emptyset$. Then, there are at most $d+1$ vectors $\{v_i\}_{i=1}^{k}\subseteq V$, $k \le d+1$ such that $\bigcup_{i=1}^{k}\{x\in\R^d \mid \langle v_i, x\rangle \leq 0\} = \R^d$.
\end{lemma}
\begin{proof}
Note that the assumption on $V$ is equivalent to 
    \[\bigcup_{v\in V}\{x\in\R^d\mid\langle -v,x\rangle \leq 0\} = \R^d.\]
By Lemma~\ref{lem:hahn_banach}, $0\in\Conv(\{-v\mid v\in V\})$. Next, by Caratheodory's Convex Hull Theorem there exist $k$ vectors, with $k \le d+1$, $\{-v_i\}_{i=1}^{k}\subseteq -V$, and $k$ non negative real numbers $\{\lambda_i\}_{i=1}^{k}$ with $\lambda_1 + \dots + \lambda_{k} = 1$ such that $0 = \sum_{i=1}^{k}\lambda_iv_i$. Then, by Gordan's Lemma 
\[\bigcap_{i=1}^{k} \{x\in\R^d \mid \langle v_i, x\rangle > 0 \} = \emptyset.\]
That is, $\bigcup_{i=1}^{k} \{x\in\R^d \mid \langle v_i, x\rangle \leq 0\}= \R^d$, as desired.
\end{proof}

\begin{lemma}
\label{lem:pasar_continuo}
    Let $V\subseteq \mathbb{S}^{d-1}$ be a closed set of directions such that $\bigcap_{v\in V}\{x\in\Z^d \mid \langle v, x\rangle< 0\} = \emptyset$. Then,
    \[\bigcap_{v\in V}\{x\in\R^d \mid \langle v, x\rangle< 0\} = \emptyset.\]
\end{lemma}

\begin{proof}
    Note that $\bigcap_{v\in V}\{x\in\Z^d \mid \langle v, x\rangle< 0\} = \emptyset$ implies that  $\bigcap_{v\in V}\{x\in\Q^d \mid \langle v, x\rangle< 0\} = \emptyset$. For the sake of a contradiction, suppose there exists $x\in\R^d$ such that $\langle x,v\rangle<0$ for all $v\in V$. Let $\{q_n\}_n\subseteq \Q^d$ such that the sequence $(q_n)_n$ converges to $x$. By assumption, there exists $v_n\in V$ with $\langle q_n,v_n\rangle\geq 0$. Passing to a subsequence, we may assume $(v_n)_n$ converges to some element  $v\in V$ because $V$ is closed. This implies $\langle x,v \rangle\geq 0$, a contradiction. 
\end{proof}

Remark that in \cref{lem:pasar_continuo} the assumption of $V$ being closed is essential. For instance, take $w$ a vector with rational independent coordinates (so that $\langle x,w\rangle \neq 0$ for all $x\in \Q^d\setminus\{0\}$), $w^{\perp}$ an orthogonal vector, and consider $V=\bigcup_{n\geq 1}\{ w - (1/n)w^{\perp}, -w-(1/n)w^{\perp}\}$. Then $w^{\perp}\in \bigcap_{v\in V}\{x\in\R^d \mid \langle v, x\rangle< 0\}$, but $\bigcap_{v\in V}\{x\in\Q^d \mid \langle v, x\rangle< 0\}=\emptyset$.

\begin{proposition}
\label{prop:descomp_Z}
    Let $(X,T,\Z^d)$ be a topological dynamical system. Then, there exist at most $d+1$ elements $\{v_i\}_{i=1}^{k}$, $k\le d+1$ in $\ND(X)$, such that $\bigcup_{i=1}^{k}\{g\in\Z^d \mid \langle v_i, g\rangle \leq 0\} = \Z^d$.
\end{proposition}

\begin{proof}
Recall that $\Z^d$ is torsion-free, that $\Z^d$ endowed with the $\ell^2$ has a symmetric horoball space, and the spaces $\ND_{\varepsilon}(X)$ are closed. Then, by Theorem~\ref{thm:vacia_tf} and Lemma~\ref{lem:pasar_continuo} we have 
    \[\bigcap_{v\in\ND(X)}\{x\in\R^d\mid\langle v,x\rangle <0\} = \emptyset.\]
By \cref{lem:Caratheodory}, there exist at most $d+1$ vectors $\{v_i\}_{i=1}^{k}\subseteq\ND(X)$ such that $\bigcup_{i=1}^{k} \{x\in\R^d \mid \langle v_i, x\rangle \leq 0\}= \R^d$. In particular, if we restrict ourselves to $\Z^d$, we obtain the desired equality. 
\end{proof} 

\subsection{Horofunctions on Heisenberg groups}
\label{ex:heisenberg}

We recall a description of the horofunctions on the Heisenberg group equipped with the Kor\'anyi distance due to \cite{Klein_Nicas_horofn_Heisenberg:2009}. An important family is that of the Busemann horofunctions. Then, thanks to the reinforced directed subspace Robinson Crusoe's theorem (\cref{thm:directed_RC_heisenberg}) we can strengthen~\cref{thm:A}: there always exist non-deterministic Busemann horofunctions and the intersection of their horoballs is trivial (\cref{thm:vacia_th_Heisenberg}).

    Let us take a look at the discrete Heisenberg groups $H_{2d+1}(\Z)$ given by
    \[H_{2d+1}(\Z) = \left\{\begin{pmatrix} 1 & \vec{v} & t \\ 0 & I_d & \vec{u}^t \\ 0 & 0 & 1\end{pmatrix}\ \Bigg\vert\ \  \vec{v},\vec{u}\in\Z^d, t\in \Z\right\},\]
    where $I_d$ is the $d\times d$ identity matrix. A particularly important case is when $d=1$. The group $H_3(\Z)$ is informally the smallest non-virtually abelian nilpotent group. It can also be defined through the finite presentation
    \[H_3(\Z) = \langle \tt{x}, \tt{y},\tt{z} \mid [\tt{x},\tt{z}], [\tt{y},\tt{z}], [\tt{x},\tt{y}]\tt{z}^{-1}\rangle.\]

    Each group $H_{2d+1}(\Z)$ in bijection  with $\Z^{2d+1}$ through the map $M_{\vec{v},\vec{u},t}\mapsto (\vec{v},\vec{u},t)$, where the latter group is endowed with the following product
    \[(\vec{v}_1,\vec{u}_1,t_1)\cdot (\vec{v}_2,\vec{u}_2,t_2) = (\vec{v}_1+\vec{v}_2,\vec{u}_1+\vec{u}_2,t_1 + t_2 + \langle \vec{v}_1, \vec{u}_2\rangle).\]
    To avoid cumbersome notation, we use this isomorphism implicitly through the notation $(\vec{v},\vec{u},t)\in H_{2d+1}(\Z)$.

    The Heisenberg groups can be endowed with the \define{Korányi gauge}~\cite{Koranyi_geo_Heisenberg:1985}, $\|\cdot\|:H_{2d+1(\Z)}\to\R$ defined by
    \[\|(\vec{v},\vec{u},t)\|^4 = \left(\|\vec{v}\|_2^2 + \|\vec{u}\|_2^2\right)^2 + (2\langle\vec{v},\vec{u}\rangle - 4t)^2.\]
    This defines the Korányi metric $\rho_K(g,f) = \|gf^{-1}\|$. In particular, this metric is proper and right-invariant under the group's right action on itself. Furthermore, the inverse map is an isometry, as a quick computation shows $\|(\vec{v},\vec{u},t)^{-1}\| = \|(\vec{v},\vec{u},t)\|$. Following~\cite{Klein_Nicas_horofn_Heisenberg:2009}, to understand the horofunction boundary of $H_{2d+1}(\Z)$ with this metric, we must look at the Korányi sphere $S^{2d}_K$ which is given by 
    \[S_K^{2d} = \{(\vec{a},\vec{b},c)\in\R^{2d+1} \mid (\|\vec{a}\|_2^2 + \|\vec{b}\|_2^2)^2 + c^2 = 1\}.\]

    Then, from~\cite[Proposition 2.8]{Klein_Nicas_horofn_Heisenberg:2009} we know the exact form of the horofunctions for the group: they are exactly the linear functionals
    \begin{align*}
        h_{(\vec{a},\vec{b},c)}(\vec{v},\vec{u},t) &= - \left\langle (\|\vec{a}\|^2+ \|\vec{b}\|^2)\vec{a} - c\vec{b}, \vec{v}\right\rangle - \left\langle (\|\vec{a}\|^2+ \|\vec{b}\|^2)\vec{b} + c\vec{a},\vec{u}\right\rangle \\
        &= \left\langle V(\vec{a},\vec{b}, c), (\vec{v}, \vec{u})\right\rangle,
    \end{align*}
    where $V(\vec{a},\vec{b},c) = -((\|\vec{a}\|^2+ \|\vec{b}\|^2)\vec{a} - c\vec{b}, (\|\vec{a}\|^2+ \|\vec{b}\|^2)\vec{b} + c\vec{a})$, for every $(\vec{a},\vec{b},c)\in S_K^{2d}$. Notice in particular that the horofunction does not depend on the value of $t$. Because it is furthermore a scalar product, $C^g(h) = h$ for all horofunctions $h$ and elements $g\in H_{2d+1}(\Z)$.  

Notice that the space of horofunctions is  identified with  the collection $\{V (\vec{a}, \vec{b}, c) \in \R^{2d}:  (\vec{a}, \vec{b}, c) \in S^{2d}_K \}$ that is an euclidean closed ball.  Hence $\partial H_{2d+1}(\Z)$ is topologically a closed ball of $\R^{2d}$. 
The collection  $\{V (\vec{a}, \vec{b}, 0) \in \R^{2d}:  (\vec{a}, \vec{b}, 0) \in S^{2d}_K \}$ is a sphere. It defines a closed set $\partial_b H_{2d+1}(\Z)$ in the space of horofunctions. This space is the collection of Busemann points \cite{Klein_Nicas_horofn_Heisenberg:2009}. 
 It is important to note that from the above description, any nonzero horofunction of $\partial H_{2d+1}(\Z)$ is a positive multiple of  an element of $\partial_b H_{2d+1}(\Z)$.
Moreover any element in $\partial_b H_{2d+1}(\Z)$ is the limit of functions $b_{g_n}$ along a sequence $g_n= (v_n,u_n,t_n)$ such that $|2\langle v_n,u_n\rangle - 4t_n|$ stays bounded.\\

The following lemma is straightforward from the above description.
  \begin{lemma}\label{lem:Koranyi}
  For every $d\geq 1$, $H_{2d+1}(\Z)$ endowed with the Korányi metric has symmetric horoballs and 
   any horofunction is subadditive.
   Moreover, the set of Busemann horofunctions  $\partial_{b} H_{2d+1}(\Z)$ has no lower bound for its horofunctions. 
  \end{lemma}
\begin{remark}\label{rem:HorofctHeisenberg}
   Observe that the set of non-deterministic horofunctions is projective in the sense that if $h$ and $\lambda h$ (for some $\lambda > 0$) are both horofunctions of $H_{2d+1}(\Z)$, then either both are in $\ND_{\ve, \pi}(X)$ or neither is, because they share the same horoball $\{h<0\}$.
\end{remark}

\subsection{Avoiding degenerate horoballs on the Heisenberg groups}

A problem when working with the Kor\'anyi metric in Heisenberg groups is the existence of a degenerated horofunction, that is, the  zero-valued constant horofunction given by $V(0,0,1) = \vec{0}$ that has a trivial horoball. To address this issue, we give a result that avoids this problem by stating any infinite action of the Heisenberg group admits non-deterministic Busemann horofunctions (\cref{thm:vacia_th_Heisenberg}). Our strategy to prove it consists in applying the directed subspace Robinson Crusoe's theorem (\cref{thm:directed_RC_heisenberg}) to specific directions (avoiding the central one that provides the degenerate horofunction) and follow the same proof as the one of \cref{thm:vacia_tf}. To do this, we need to show the central direction is never repulsed for the diagonal action, that is: it always admits  a family of pair of  asymptotic  points invariant under the action of an arbitrary element of the group. This is done in \cref{prop:dirigido_H} by studying the subaction of the center. 

We will prove the following.
\begin{proposition}
\label{prop:dirigido_H}
    Consider two topological dynamical systems $(X,T,H_{2d+1}(\Z))$ and $(Y,S,H_{2d+1}(\Z))$, as well as a topological factor $\pi:X\to Y$ that is not bounded-to-one. Then, for every $k\in H_{2d+1}(\Z)\setminus Z(H_{2d+1}(\Z))$ and $\ve>0$, there exists $x\neq y$ in $X$ such that $\pi(x) = \pi(y)$ and 
    \[\sup_{i\in\Z}d\left(T^{z^ik^{-n}}x,T^{z^ik^{-n}}y\right)\leq\ve,
    \quad \text{for all } n\geq 1.
    \]
\end{proposition}

Of independent interest, as a consequence of \cref{prop:dirigido_H},  we get that the subaction of the center is never expansive for the Heisenberg group. This is  a new  proof. Note that this was proved before in~\cite{bitar2022distortion} and more recently in~\cite{prusik2026expansiveness} with different methods.

\begin{corollary}\label{cor=ZnonExp}
    For any infinite topological dynamical system $(X,T,H_{2d+1}(\Z))$, the subaction of its center $Z$ is non-expansive.
\end{corollary}
\begin{proof} It suffices to apply \cref{prop:dirigido_H} with the space $Y$ being a singleton to contradict the definition of expansive action.
\end{proof}

\begin{proof}[Proof of \cref{prop:dirigido_H}] 
For simplicity, fix $d\geq 1$ and denote $H = H_{2d+1}(\Z)$.
Let $z$ denote the generator of the center of $H$, $Z = Z(H)$ and consider an element $k\in H\setminus Z$. 
Let $K$  denote the abelian group generated by $k$. Recall from the notations of \cref{sec:ExpSub}, the group generated by $A_K$ is then a finite index subgroup of $Z$, and the group $H_K$ generated by $K$ and $A_K$ is isomorphic to $\Z^2$. In particular, this isomorphism enables us to define a scalar product on $H_K$.  Let $-\vec{w}$ be the orthogonal projection of $k$ on  $\ker \pi_A$. It holds that $\langle \vec{w}, k^{n}\rangle \le 0$ for all $n>0$. \cref{thm:nonexpansive_subgroup} provides a vector $u\in \ker \pi_A \cap \{\langle \vec w, u \rangle \le 0\}$, and for each $\ve>0$ two distinct points $x_\ve,y_\ve\in X$ such that $\pi(x_\ve)=\pi(y_\ve)$, and 
$d(T^{g} x_\ve, T^g y_\ve) \le \ve$ for all $g$ in $\{g\in H_K\mid \langle u,g\rangle <0\}$.  
By the very choice of $u$, this former set is invariant by translation by  the elements  $k^{-n}$, for any $n>0$ and by translation by $A_K$-elements. This shows the result.
\end{proof}

\subsubsection{Non-deterministic Busemann horofunction for the Heisenberg group}

We can now prove the main theorem of this section. We will use the notation $H= H_{2d+1}(\Z)$.
\begin{theorem}\label{thm:vacia_th_Heisenberg}
Let $(X,T,H)$ and  $(Y,S,H)$ be topological dynamical systems. Let $\pi \colon X \to Y$ be a factor map that is not bounded-to-one. Then, $\ND_{\varepsilon, \pi}(X)\cap\partial_b H$ is closed  and non-empty for every $\ve>0$. In particular $\ND_{\pi}(X)\cap\partial_b H$ is a non-empty closed set. 

Moreover, it holds that
\[\bigcap_{h\in\ND_\pi(X)\cap\partial_b H} \{h<0\} = \emptyset.\]
\end{theorem}

Notice that when $(Y,S,H)$ is the trivial one point system and $X$ is infinite, there exist at least two non-deterministic Busemann horofunctions (necessarily non-trivial).

\begin{proof}
From the description of the horofunctions for the Korányi metric, the set $\partial_b H$ of Busemann horofunctions  is closed. Then, we get $\ND_{\varepsilon, \pi}(X)\cap\partial_b H$ is closed from \Cref{thm:closeness-ND} thanks to the description of the horofunctions from~\Cref{lem:Koranyi}. To show this intersection is not empty is more tricky.

We introduce the following notations: set $Z$ to be the center of $H$, and set $L$ to be the set 
\[L = \left\{(v,u,t)\in H \mid t = \left\lfloor\frac{\langle v,u\rangle}{2}\right\rfloor\right\}.\]

Let us see that $L$ is a left transversal of $Z$. Indeed, for $z=(0,0,s)$ one has $(v,u,t)z=(v,u,t+s)$, so each coset $gZ$ is the vertical line over the horizontal part of $g$, and $L$ picks exactly one point of each. Next, denote $\lambda=\|v\|^2+\|u\|^2$ and $\tau(t)=2\langle v,u\rangle-4t$. Now,
$\|(v,u,t)\|^4=\lambda^2+\tau(t)^2$. As $t$ ranges over
$\Z$, $\tau$ ranges over $2\langle v,u\rangle+4\mathbb Z$. So $|\tau|$ is minimized at $0$ when $\langle v,u\rangle$ is even, and at $2$ when it is odd. Then, $t=\lfloor\langle v,u\rangle/2\rfloor$ attains that minimum in both cases. Therefore, $\rho(g,1_G)=\inf_{z\in Z}\rho(g,z)$ for every $g\in L$. Thus, $(L,Z)$ is an asymptotically orthogonal decomposition of $H$. Finally, since $(v,u,t)^{-1} = (-v,-u,\langle u,v\rangle-t)$, we have $L^{-1} = \left\{(-v,-u,\left\lceil\frac{\langle u,v\rangle}{2}\right\rceil)\mid u,v\in\Z^{d}\right\}$. Finally, $\partial L = \partial L^{-1} =\partial_b H_{2d+1}(\Z)$.  We will apply~\Cref{thm:directed_RC_heisenberg} using this decomposition.\\

Like in the proof of \cref{thm:vacia_tf}, to show that the intersection of their horoballs is empty, it is enough to prove that 
\begin{align}\label{eq:claim}
\ND_{\varepsilon, \pi}(X)\cap \partial_{b} H \cap \HB(k) \neq \emptyset, \textrm{ for any } k \in H \setminus Z.
\end{align}
Indeed,  this directly implies that
\[\bigcap_{h \in \ND_{\varepsilon, \pi}(X)\cap \partial_{b} H} \{h<0\} \subseteq Z.
\]
Recall, by the description of the horofunctions of $H_{2d+1}(\Z)$, that no horofunction is negative on the center. Hence this shows that the intersection is in fact empty.  In particular, there exist at least two   non-deterministic Busemann horofunctions. This shows the theorem. So, we are reduced to  prove \eqref{eq:claim} similarly to the one of \cref{theo:factors_general_cone}. 

We consider the system $(R_{\pi}, T^{(2)}, H)$, where $R_\pi=\{(x,y)\in X\times X: \pi(x)=\pi(y)\}$ and $T^{(2)}$ is the diagonal action $T^{(2)}_g(x,y)= (T_g(x),T_g(y))$.
Let  $F$ denote  $\triangle_X$ the diagonal on $X$.
The set $F$ is clearly closed and $H$-invariant. It is also not isolated. To see this, note that since $\pi$ is not bounded to one,  there exists a sequence of distinct points $x_n$ accumulating at $x$ with $(x,x_n)\in R_{\pi}$ for all $n$. That means that  $F$ is not isolated in  $R_{\pi}$. 

From the description of horofunctions of $H_{2d+1}(\Z)$, there exists $\bar{h}\in \partial_b H$ such that $\bar{h}(k^{-1})<0$. Define, for $0<\eta<-\bar{h}(k^{-1})$, the set \[H_\eta=\{ g\in L \mid  \rho(k^{-1}, g) < \rho(1_G, g)-\eta\}.\] 

Let $g_n\in L$ such that $b_{g_n}\to \bar{h}$. Note that $H_\eta$ is unbounded since it contains all but finitely many $g_n$'s, and by definition 
\[k^{-1}\in \bigcap_{h\in \partial H_{\eta}} \{h\leq -\eta\} \subseteq \bigcap_{h\in \partial H_{\eta}} \{h<0\}.\]

Also note that $L\setminus H_{\eta}$ is unbounded, otherwise, we would have $\partial L=\partial H_{\eta}$ and so $h(k^{-1}) <0$ for every $h \in \partial  L$. But, $\partial L = \partial_b H$, so $h(k^{-1})<0$ for every $h\in\partial_b H$. Since $k\notin Z$, the horizontal part of $k^{-1}$ is non-zero. Then, if we denote by $h_V$ the horofunction given by the vector $V$, $h_{V}(k^{-1})$ and $h_{-V}(k^{-1})$ are non-zero of opposite signs, which is a contradiction with the description of the Busemann horofunctions. 

We claim that the diagonal $F$ is not $\{k^{-1}\}$-repulsive, that is for all $\delta >0$, there is a $x\in F^c$ and infinitely many $g \in \langle k^{-1}\rangle_+$ such that $\sup_{z\in Z} d(T^{zg}x, F) \le \delta$. This is the case by~\Cref{prop:dirigido_H}.

Fix $\ve>0$. Applying the directed subspace  Robinson Crusoe Theorem \ref{thm:directed_RC_heisenberg} to $F$ and $G_0=H_{\eta}$, we obtain a horofunction in $\partial [L\setminus H_{\eta}]^{-1}\cap \ND_{\ve,\pi}(X)$. Since $L\setminus H_{\eta}\subseteq L$, we have $\partial[L\setminus H_{\eta}]^{-1}\subseteq \partial L^{-1} = \partial_b H$. Furthermore, $[L\setminus H_{\eta}]^{-1}\subseteq H\setminus H_{\eta}^{-1}$ so $\partial[L\setminus H_{\eta}]^{-1}\subseteq \partial[H\setminus H_{\eta}^{-1}]$. Therefore, $\partial [H\setminus H_{\eta}^{-1}]\cap\partial_{b}H\cap \ND_{\ve,\pi}(X)$ is not empty. 

By compactness, there exist $h\in \partial_{b}H$ and a sequence $\eta_j\to 0$ with $h_{\eta_j}\in \partial [H\setminus H_{\eta_j}^{-1}]\cap\partial_{b} H \cap \ND_{\ve,\pi}(X)$ such that $h_{\eta_j}\to  h$. Since  $\ND_{\ve,\pi}(X) \cap \partial_{b}H$ is closed, the limit $h$ belongs to $\ND_{\ve,\pi}(X)$. Noting that the sets $(H\setminus H_{\eta}^{-1})$ decrease as $\eta\to 0$, we conclude that $h\in \bigcap_{\eta>0} \partial[H\setminus H_{\eta}^{-1}] \cap \ND_{\ve,\pi}(X)$. 

By a diagonal argument, we may write  $h=\lim_{j\to \infty} b_{g_j}$ for a sequence $(g_j)_{j\in \N}$ going to infinity and such that  $g_j\in H\setminus H_{\eta_j}^{-1}$. Let $h'=\lim b_{g_j^{-1}}$, which exists after taking a subsequence, if needed. Since $g_j^{-1} \notin H_{\eta_j}$ we have that for all $j\in \N$, $\rho(k^{-1},g_j^{-1})\geq \rho(1_G,g_j^{-1})-\eta_j$, which means that $h'(k^{-1})\geq 0$. 
By the symmetric horoballs space assumption, we obtain $h(k)\geq 0$. That is, $h\in \Cone(k)\cap \ND_{\ve,\pi}(X) \cap \partial_{b}H$, as desired.
        
\end{proof}



\subsection{Covering results for the Heisenberg group}
We can also prove a result analogous to~\Cref{prop:descomp_Z} for the Heisenberg groups. To do this notice that $H_{2d+1}(\Z)$ is embedded in the Lie group $H_{2d+1}(\R)$, which itself is in bijection with $\R^{2d+1}$ with the aforementioned map and product.

\begin{proposition}
\label{prop:descomp_H}
    Let  $(X,T,H_{2d+1}(\Z))$ be a topological dynamical system. Then, there exist at most $2d+1$ elements $\{(\vec{a}_i,\vec{b}_i,c_i)\}_{i=1}^{k}$ in $\ND(X)\cap  \partial_{b}H_{2d+1}(\Z)$
    , $k\le 2d+1$ such that 
    \[\bigcup_{i=1}^{k} \left\{(\vec{v},\vec{u},t)\in H_{2d+1}(\Z) \mid \langle V(\vec{a}_i,\vec{b}_i,c_i), (\vec{v},\vec{u})\rangle \leq0\right\} = H_{2d+1}(\Z).\]
\end{proposition}

\begin{proof}
    Given that $(H_{2d+1}(\Z), \rho_K)$ has symmetric horoballs and any horofunction is subadditive, \cref{thm:closeness-ND} and \cref{thm:vacia_th_Heisenberg} imply that  the space $\ND(X)\cap \partial_{b} H_{2d+1}(\Z)$ is non-empty and closed. Denote this set $\ND(X)_{b}$. Moreover,  we have  
    \[\bigcap_{(\vec{a},\vec{b},c)\in\ND(X)_{b}}\left\{(\vec{v},\vec{u},t)\in H_{2d+1}(\Z) \mid \langle V(\vec{a},\vec{b},c), (\vec{v},\vec{u})\rangle < 0\right\} = \emptyset.\]
    As mentioned before, notice that the last coordinate, namely $t$, does not intervene in the inner product. Therefore, the empty intersection is equivalent to
    \[\bigcap_{(\vec{a},\vec{b},c)\in\ND(X)_{b}}\left\{(\vec{v},\vec{u})\in \Z^{2d} \mid \langle V(\vec{a},\vec{b},c), (\vec{v},\vec{u})\rangle < 0\right\} = \emptyset.\]
    In particular, by Lemma~\ref{lem:pasar_continuo} we can extend this empty intersection for elements of $\R^{2d}$, that is,

    \[\bigcap_{(\vec{a},\vec{b},c)\in\ND(X)_{b}}\left\{(\vec{x},\vec{y})\in \R^{2d} \mid \langle V(\vec{a},\vec{b},c), (\vec{x},\vec{y})\rangle < 0\right\} = \emptyset. \]
    
    Then by Lemma~\ref{lem:Caratheodory}, there exist at most $2d+1$ elements $\{(\vec{a}_i,\vec{b}_i,c_i)\}_{i=1}^{k}\subseteq \ND(X)_{b}$ such that $\bigcup_{i=1}^k\{(\vec{x},\vec{y})\in \R^{2d} \mid \langle V(\vec{a}_i,\vec{b}_i,c_i), (\vec{x},\vec{y})\rangle \leq 0\} = \R^{2d}$. This implies that $\bigcup_{i=1}^k\{(\vec{x},\vec{y}, z)\in \R^{2d+1} \mid \langle V(\vec{a}_i,\vec{b}_i,c_i), (\vec{x},\vec{y})\rangle \leq 0\} = \R^{2d+1}$.
    
    By the aforementioned bijection between $H_{2d+1}(\R)$ and $\R^{2d+1}$ we can view this union as
    \[\bigcup_{i= 1}^k\left\{(\vec{x},\vec{y},z)\in H_{2d+1}(\R) \mid \langle V(\vec{a},\vec{b},c), (\vec{x},\vec{y})\rangle \le 0\right\} = H_{2d+1}(\R).\] 
    Restricting this equality to elements of $H_{2d+1} (\Z)$, allows us to conclude.
\end{proof}


\subsection{Horofunctions on free groups}\label{sec:FreeGroup}

Let us look at an interesting case where the uniform quasi-subadditive condition does not hold. Fix $n\in\N$ and consider the free group $\F_n$ with a free generating set $S = \{\tt{s}_1,...,\tt{s}_n\}$. This generating set induces a word metric $\rho_S$ on $\F_n$. The horofunction boundary for this metric is well understood~\cite{Grigorchuck_Kaimanovich_Nagnibeda_ergodic_boundary_actions_NS:2012}, it is homeomorphic to the subshift $X_r\subseteq (S\cup S^{-1})^{-\N}$ defined by
\[X_r = \{\xi\in (S\cup S^{-1})^{-\N} \mid \forall w\factor \xi, \ w \textnormal{ is a reduced word }.\}\]
Then, for every $h\in \partial\F_n$ there exists $\xi\in X_r$ such that 
\[h(g) = h_\xi(g) = |g| - 2(g|\xi),\]
where $(\cdot|\cdot)$ is the Gromov product, which in this case is equal to the length of the longest common suffix between $\xi$ and the unique geodesic describing $g$. We also have an explicit description for horoballs for this metric. The horoball associated to $\xi\in X_r$ is given by the set of all elements $g\in\F_n$ such that there exists $m\in\N$ with $w_{[0,m-1]} = \xi_{[-(m-1), 0]}$ and $|w|<2m$ where $w$ is the geodesic that defines $g$.

With this representation, we can better understand the topological dynamical system $(\partial\F_n,C,\F_n)$. For a word $w\in (S\cup S^{-1})^*$ and $\xi\in X_r$, we denote by $\red(\xi w)$ the word we obtain after concatenating $\xi$ and $w$ and reducing pairs $ss^{-1}$. Then, the action of an element $g\in\F_n$ is given by $C^g(h_\xi) = h_{\red(\xi w)}$ where $w$ is the unique geodesic defining $g$. Furthermore, the action is minimal~\cite{Grigorchuck_Kaimanovich_Nagnibeda_ergodic_boundary_actions_NS:2012}, and is conjugate to an $\F_n$-SFT~\cite{Coornaert_Papadopoulos_symb_dyn_hyperbolic_groups:1993}.

\begin{remark}
The horofunction boundary of $\F_n$ endowed with the word metric $\rho_S$ is not uniformly quasi-subadditive. Consider $n\geq 2$, two distinct generators $\tt{s}, \tt{t}\in S$, and the point $\xi = \tt{s}^{\infty}$. We have $h_{\xi}(\tt{s}^n\tt{t}) = n+1$, $h_{\xi}(\tt{s}^n) = -n$, and $h_{\xi}(\tt{t}) = 1$. Thus, 
    \[h_{\xi}(\tt{s^{n}\tt{t}}) > h_{\xi}(\tt{s}^n) + h_{\xi}(\tt{t}) + n.\]
    Therefore, the boundary is not uniformly quasi-subadditive.
\end{remark}

\begin{proposition}\label{prop:densityNDHorofunction}
   Let  $(X,T,\F_n)$ be a topological dynamical system, and endow $\F_n$ with the word metric for the free generating set. Then, for every $\varepsilon>0$ we have $\overline{\ND_{\varepsilon}(X)} = \partial\F_n$.
\end{proposition}

\begin{proof}
Consider a horofunction given by $\xi\in X_r$. Furthermore, suppose the action is expansive (if not $\ND(X) = \partial G$). Now, take any non-deterministic horofunction $h_{\zeta}$ with $\zeta\in X_r$, along with a $(h_\zeta, \ve)$-asymptotic pair $x,y\in X$. By definition we have,
    \[\sup_{f\in \{h_{\zeta}<0\}}d(T^fx,T^fy)\leq \varepsilon.\]
    
For every $n\in \N$, define $w_n = \xi_{[-(n-1),0]}$, $v_n = w_n^{-1}\zeta_{[-2n-1,0]}$. Consider the points $x_n = T^{v_n}x$, $y_n = T^{v_n}y$. Notice that for all $n\in\N$, we have $h_{\xi}(v_n^{-1}) \leq -n-1$. Let us show that for every $n\in\N$ the horofunction defined by $\xi_n = \zeta_{[-\infty,-2n-1]}w_n$ is $\ve$-non-deterministic. Indeed, if we consider $w\in \{h_{\xi_n}<0\}$ we have two cases. First, if $w$ is of the form $w = u\zeta_{[-k,-2n-1]}w_n$ for some $k>2n+1$ and $0\leq |u| < k-n-1$, then we have,
\[d(T^{w}x_n, T^w y_n) = d(T^{u\zeta_{[-k,0]}}x, T^{u\zeta_{[-k,0]}}y)\leq \varepsilon,\]
as $u\zeta_{[-k,0]}\in\{h_{\zeta}<0\}$. The second case is when $w = uv$ where $v$ is a suffix of $w_n$, and $0\leq|u|<|v|<n$. Then, $|ww_n^{-1}| \leq |u| + n - |v| < n$ and thus $ww_n^{-1}\zeta_{[-2n-1,0]}\in\{h_{\zeta}<0\}$. Finally, this implies 
\[d(T^{w}x_n, T^w y_n) = d(T^{ww_n^{-1}\zeta_{[-k,0]}}x, T^{ww_n^{-1}\zeta_{[-k,0]}}y)\leq \varepsilon.\]
Therefore $(x_n, y_n)$ is an $(h_{\xi_n}, \ve)$-asymptotic pair. We conclude by noticing that $\xi_n$ converges to $\xi$, and therefore the space $\ND_{\ve}(X)$ is dense in $\partial\F_n$.
\end{proof}

\section{Applications to topological minimal self-joinings and double minimality} \label{sec:TMSJ}

 The notion of topological minimal self-joining was introduced by del Junco~\cite{delJunco_minimal_self_joining_top_dyn:1987} as a topological analog of the measurable concept of minimal self-joining. Here, we present the principal properties of such systems. We then provide an example of a (strong) doubly minimal system for $\Z^d$. Finally, in Section \ref{sec:LimitnFolding}, we use the results obtained in previous sections to establish a bound on the fold of the topological minimal self-joining, depending on the rank of the abelian group or that of the Heisenberg group.

\begin{definition}\label{def:TSMJ}
Let $n\in \N$. We say a topological dynamical system $(X,T,G)$ has \define{$n$-fold strong topological minimal self-joining} ($n$-fold sTMSJ for short) if $X$ is infinite and for every $n$ distinct points $x_1,\dots,x_n\in X$ all lying on different $Z(G)$-orbits\footnote{Here, $Z(G)$ denotes the center of $G$, that is, the group of elements that commute with all elements of $G$}, the orbit of $(x_1,\ \dots,\ x_n)$ is dense in $X^n$.
Following the terminology of \cite{Weiss_doubly_minimal:1995}, if  $(X,T,G)$ has $2$-fold sTMSJ, we may also say that it is \define{strong doubly minimal}. Furthermore, if $(X,T,G)$ has $n$-fold sTMSJ for all $n\in\N$, we say it has $\infty$-fold sTMSJ.
\end{definition}
 The family of strong doubly minimal $\Z$-systems is very rich: there is no restriction on their ergodic properties beyond zero entropy. Weiss showed \cite{Weiss_doubly_minimal:1995} that any zero entropy ergodic system has a strictly ergodic model which is strong doubly minimal.

\begin{remark}\label{rem:+=>-}
If $(X,T,G)$ has $n$-fold sTMSJ, then it has $k$-fold sTMSJ for all $k\leq n$. Also note that $1$-fold sTMSJ means minimality of $(X,T,G)$. 
\end{remark}

\begin{remark}
\label{rem:true_doubly_minimal}
    The proposed definition of topological minimal self-joinings on groups is quite strong, as it name suggests. Indeed, suppose $G$ admits a strong doubly minimal action on $X$. Given an element $g\in G$, consider its centralizer $C_G(g) = \{f\in G \mid fgf^{-1}=g\}$. If $C_G(g)\neq G$ and $[G:C_G(g)]<+\infty$, then the diagonal action of $T$ on the graph of $T^g$ in $X^2$ has finite orbit. If $L$ is a finite set of coset representatives for $C_G(g)$, the union 
    \[\bigcup_{l\in L}T^l\{(x,T^gx) \mid x\in X\},\]
    is a closed $G$-invariant subset of $X^2$. This would contradict strong double minimality. Notice that an element $g\in G$ has a finite index centralizer if and only if its conjugacy class is finite. If we denote by $\textnormal{FC}(G)$ the union of all finite conjugacy classes, any group such that $\textnormal{FC}(G)\setminus Z(G)\neq\emptyset$ does not admit a strong doubly minimal system.

    We say a system has $n$-fold topological minimal self-joining if $X$ is infinite and for every $n$ distinct points $x_1,\dots,x_n\in X$ all lying on different \emph{$G$-orbits}, the orbit of $(x_1,\ \dots,\ x_n)$ under the diagonal action is dense in $X^n$. Of course, both definitions coincide when $\textnormal{FC}(G) \subseteq Z(G)$, which is the case for torsion-free nilpotent groups~\cite[Lemma 3.2]{eckhardt2018c} and ICC groups (where $\textnormal{FC}(G) = \{1_G\}$). We leave the study of the latter notion of topological minimal self-joinings to a future work.
\end{remark}

It is worth noting that the original definition in \cite{King_top_minimal_self_joinings:1990, Weiss_doubly_minimal:1995} requires the system to be totally minimal. However, the next proposition (\cref{prop:totally_minimal}) shows that this assumption is directly implied by the denseness of tuples.

\begin{proposition} \label{prop:totally_minimal}
     Let $(X,T,G)$ be a topological dynamical system that has $n$-fold sTMSJ, with $n \ge 2$.  Then $(X,T,G)$ is totally minimal, that is, $(X,T,H)$ is minimal for any finite index subgroup $H\leqslant G$.    
\end{proposition}

\begin{proof}
Assume that $(X,T,H)$ is not minimal. It is a classical algebraic fact (Poincar\'e's Theorem) that there exists a subgroup $H'\leqslant H$ that is normal and of finite index in $G$. It is clear that the topological dynamical system $(X,T,H')$ is not minimal either.
Let $X' \subsetneq X$ be a $H'$-minimal subset. Let $K \subseteq G$ be a finite set of representatives for the left $H'$-cosets. The set $\bigcup_{g \in  K } T^g X' $ is a closed subset of $X$, and by normality of $H'$ it is $G$-invariant. Hence it is all of $X$ by minimality of $(X,T,G)$. Note that $T^g(Y)$ is $H'$-invariant for all $g\in G$ and all $H'$-invariant subset $Y$. From this, using the minimality of $(X',T,H')$, we find that the sets $T^g(X')$, $g\in K$, are either equal or disjoint. Since $X$ is the finite disjoint union of these closed sets, it follows that $X'$ is clopen and infinite. As a minimal set, it has to be perfect, hence uncountable. So, we may pick $x,y \in X'$ in different $Z(G)$-orbit.  Since $(X,T,G)$ has $2$-fold sTMSJ, there is a $g \in G$ such that $T^g (x) \in X'$ and   $T^g(y) \notin X' $. This is impossible since it would imply $T^{g^{-1}}(X')\cap X'\neq \emptyset$ and $T^{g^{-1}}(X')\neq X'$.     
\end{proof}

\subsection{General properties of strong doubly minimal systems}
In this section, we review several properties of double minimal systems. Most of these are already well known in the context of integer actions, and their proofs follow along the same lines as the integer case.

Let $(X,T,G)$ be a topological dynamical system. A homeomorphism $\phi\colon X\to X$ is an \define{automorphism} of $(X,T,G)$ if $\phi\circ T^g=T^g\circ \phi$ for all $g\in G$. We let $\Aut(X,T,G)$ denote the set of automorphisms of $(X,T,G)$. 
Note that $T^g \in \Aut(X,T,G)$ for any element $g\in Z(G)$.  

A topological dynamical system $(X,T,G)$ is \define{coalescent} if every continuous, surjective map $\phi\colon X\to X$ such that $\phi\circ T^g=T^g\circ \phi$ for all $g\in G$ is automatically injective, hence an automorphism.
\begin{lemma} \label{lem:coalescent}
    Let $(X,T,G)$ be a strong doubly minimal system. Then, it is coalescent and any automorphism is of the form $T^g$ for some  $g\in Z(G)$.
\end{lemma}
\begin{proof}
   Let $\phi \colon X \to X$ be a continuous onto map such that $\phi\circ T^g=T^g\circ \phi$ for all $g\in G$. The set $\{(x,\phi(x)) \mid x\in X\}$ is closed, $G$-invariant and not equal to $X\times X$. So, for (any) $x\in X$ there exists $g\in Z(G)$ such that $\phi(x)=T^{g}(x)$. The minimality of $(X,T,G)$ implies that $\phi=T^g$.
\end{proof}

The next result appears in \cite{King_top_minimal_self_joinings:1990}. Although is written for the case of $G = \Z$, the same proof works in general, as is just a property about continuous commuting maps.

\begin{lemma}[\cite{King_top_minimal_self_joinings:1990}]
\label{lem:uncountable_aut}
    If there exists a sequence $(P_n)_{n\in\N}\in\Aut(X,T,G)$ that commute with each other, $\sup_{x \in X} d(P_n(x), x)\to 0$ and $P_n\neq id$, then for every $\varepsilon>0$ the ball 
    \[\{S\in\Aut(X,T,G) \mid \sup_{x \in X} d(S(x), x) \leq\varepsilon\},\]
    is uncountable.
\end{lemma}

The following lemma can be proved using the same idea of \cref{lem:coalescent} and has appeared in several other articles (for instance \cite{Auslander_endomorphisms_minimal_sets:1963,Donoso_Durand_Maass_Petite_automorphism_low_complexity:2016}). 
\begin{lemma}\label{lem:min+faithfull=>free}
Let $(X,T,G)$ be a minimal topological dynamical system. Then the action of $\Aut(X,T,G)$ on $X$ is free. 
\end{lemma}

We use this to prove the following generalization of the result by Huang and Ye~\cite{Huang_Ye_double_minimal:2015}.

\begin{proposition}
\label{prop:expansive}
    Let $G$ be a countable group and $(X,T,G)$ a strong doubly minimal topological dynamical system. Then $(X,T,G)$ is expansive.
\end{proposition}

\begin{proof}
    Note that \cref{lem:uncountable_aut,lem:coalescent} imply that there exists $\epsilon>0$ such that if $\phi \in \Aut(X,T,G)$ is such that $\sup_{x\in X} d(\phi(x), x) \leq \epsilon$, then $\phi={\rm Id}$. We claim that $(X,T,G)$ is expansive with constant $\epsilon$. Indeed, if $x,y\in X$ are such that $d(T^gx,T^gy)\leq \epsilon$ for all $g\in G$, then as $(X,T,G)$ is strong doubly minimal, necessarily $y=T^{g'}x$ for $g'\in Z(G)$. But the minimality of $(X,T,G)$ implies that $\sup_{z\in X} d(T^{g'}(z), z)\leq \epsilon$, hence $T^{g'}={\rm Id}$, and $x=y$. 
\end{proof}

Given a subset $A\subseteq G$ and $x\in X$, we let $\omega_A(x)$ denote the set of points $y\in X$ such that $y=\lim T^{g_i}x$, for some sequence $g_i$ in $A$.

\begin{lemma}
\label{lem:minimal_por_horobola}
    Let $(X,T,G)$ be a minimal topological dynamical system, and $G_0\subseteq G$ a thick set. Then, for every $x\in X$ we have $\omega_{G_0}(x)=X$. In particular, for a horofunction $h\in\partial G$ such that $\inf_{g\in G}h(g)=-\infty$ and any $x\in X$, we have $\omega_{\{h<0\}}(x)=X$.
\end{lemma}
\begin{proof}
    Let $x\in X$ and $U\subseteq X$ be a nonempty open set. Because $(X,T,G)$ is minimal, the set
    $A = \{g\in G \mid T^gx\in U\}$ is syndetic. Therefore,  $A\cap G_0\neq \emptyset$, from where it follows that $\omega_{G_0}(x)=X$. The fact that $\{h<0\}$ is thick if $\inf_{g\in G}h(g)=-\infty$ was proved in \cite[Proposition 2.1]{Donoso_Maass_Petite_geometric_asymptotic:2024}.
\end{proof}

Using similar ideas, in the context of strong doubly minimal systems, we can prove

\begin{lemma}
\label{lem:asym_in_doubly_faithful}
 Let $(X,T,G)$ be a strong doubly minimal infinite topological dynamical system. Assume that $(G, \rho)$ has thick horoballs. Then, there exists a proximal pair $(x,y)\in P$ where $x$ and $y$ lie on different $Z(G)$-orbits.
\end{lemma}
\begin{proof}
Using \cref{thm:distal_expansive}, we get that there exist $x\neq y$ with $(x,y)\in P$. If $y=T^g x$, with $g\in Z(G)$, then there exists $z\in X$ such that $T^g z=z$. By minimality of $(X,T,G)$ we deduce that $T^g={\rm id}$, which contradicts that $x\neq y$.   
\end{proof}

\begin{proposition}\label{prop:factor_trivial}
    Let $(X,T,G)$ be a strong doubly minimal system where $G$ is a finitely generated group. Then, the following hold:
    \begin{enumerate}
        \item If $Z(G)=\{1_G\}$, then  the system is prime, i.e. all strict factors are trivial.
        \item All strict factors are not $Z(G)$-faithful.
        \item Any not bounded-to-1 factor of $(X,T,G)$ is trivial. In particular, if $Z(G)$ is torsion-free and acts faithfully, then all strict factors are trivial.
    \end{enumerate}
\end{proposition}
\begin{proof}
Let $(Y,S,G)$ be a topological dynamical system and $\pi\colon X\to Y$ be a strict factor. Let $R$ be the equivalence relation on $X\times X$ induced by $\pi$, that is $R= \{(x,y) \in X^2 \mid \pi(x) = \pi(y)\}$ and for $x\in X$ denote $R[x]=\{y:(x,y)\in R\}=\pi^{-1}(\pi(x))$.

Note that if $Y$ is not trivial, then for all $x\in X$, we have $R[x]\subseteq \{T^gx: g \in Z(G)\}$, since otherwise the strong double minimality of $(X,T,G)$ would imply $R=X\times X$ and $Y$ would be trivial. We obtain $(1)$. 

Note that there exists $g\in Z(G)$, $g\neq 1_G$ and $x\in X$ such that $(x,T^gx)\in R$. Using \cref{lem:min+faithfull=>free}, we get that $S^g$ is the identity map in $Y$, concluding $(2)$. 

To see $(3)$, suppose that the factor $\pi$ is not bounded-to-1 and $Y$ is not trivial. By \cref{thm:distal_expansive} there exists $(x,y)\in R\cap P$ with $x\neq y$. as $Y$ is not trivial, there exists $g\in Z(G)$ such that $y=T^gx$. As $(x,T^gx)\in P$, we get $T^gz=z$ for some $z\in X$. This implies that $T^g={\rm Id}$, contradicting that $x\neq y$.

Recall by $(2)$ that there exists $g\in Z(G)$ such that $(x,T^{g^n}x)\in R$ for all $n\in\N$. If $G$ is torsion-free and acts faithfully, all the points $T^{g^n}x$ are distinct, and hence the factor map is not bounded-to-1. We get that $Y$ is trivial. 
 \end{proof}

Keynes and Newton showed in \cite[Proposition 2.11]{Keynes_Newton_prime_flows:1976} that if $(X,T,G)$ is prime (that is, all strict factors are trivial), then $X$ is zero-dimensional or connected \footnote{Although many results in that section of their paper were stated for abelian groups, the proof of this result works for any group.}. As a direct corollary of \Cref{prop:factor_trivial} we get.
 \begin{corollary} \label{cor:doubly_min_dimension}
Let $G$ be finitely generated and let $(X,T,G)$ be a strong doubly minimal system such that $Z(G)$ is torsion-free and acts faithfully on $X$. Then $X$ is zero-dimensional or connected.
 \end{corollary}

It is worth noting that for $G=\Z$, only the zero-dimensional case is possible in \cref{cor:doubly_min_dimension}. This follows from the celebrated theorem of Ma\~n\'e \cite{Mane_expansive_and_dimension:1979} on the dimension of minimal expansive systems. For other groups, both cases are possible. We describe a (strong) doubly minimal system for $\Z^d$-action on the Cantor set in \cref{theo:KingExample}, and we describe below a doubly minimal action on the circle for the free group on two generators.

\begin{example}
\label{ex:dm_connected}
    In this example, many arguments are classical. We refer to  \cite[Chapter 3]{Zimmer_erg_theory_semisimple_groups:1984} for details. Consider the following two matrices 
    $A$ (hyperbolic) and $B$ (elliptic) in ${\rm SL}_2(\R)$:
    \[ A= \begin{pmatrix}2 & 0 \\ 0 & 1/2  \end{pmatrix} \textrm{ and } B= \begin{pmatrix} \cos \theta & \sin \theta \\ -\sin \theta& \cos \theta \end{pmatrix},
        \]
for some real $\theta \notin {\mathbb Q} \pi$. Since there are only countably many polynomials with integer coefficients, there is a $\theta$ so that  the group $\langle A, B \rangle$ generated by $A$ and $B$ is free (indeed, any possible relation between $A$ and $B$ implies polynomials relation between $\cos \theta$ and $\sin \theta$) and Zariski dense in ${\rm SL}_2(\R)$.
There are some rigidity results on $ {\rm SL}_2(\R)$: a Zariski dense  subgroup in $ {\rm SL}_2(\R)$ is either dense or discrete (indeed the Lie algebra of its closure is an ideal of the simple Lie algebra of ${\rm SL}_2(\R)$, hence is trivial or is everything).
As a consequence, since the group    $\langle A, B \rangle$ is not discrete (the matrices $B^n$ are close to the identity for suitable integers $n$),  it is dense in ${\rm SL}_2(\R)$. Recall that the action of  ${\rm SL}_2(\R)$ by M\"obius transformations on the circle $\Cercle^1$ is $2$-transitive. It is then straightforward to check that the (sub-)action of $\langle A, B \rangle$ on  $\Cercle^1$ is doubly minimal. 
\end{example}

\subsubsection{Further remarks regarding entropy} In this subsection, we explore connections to entropy theory. While there are many other natural questions in this direction, we omit them here to avoid deviating from the main scope of the paper.
\begin{proposition} \label{prop:doublymin_zero_entropy}
    Let $(X,T,G)$ be a doubly minimal system where $G$ is a countable amenable group. Then it has zero topological entropy $h_{\topo}(X,T,G)=0$. 
    \end{proposition}

\begin{proof}
Suppose $h_{\topo}(X,T,G)>0$. By the variational principle for countable amenable groups (see~\cite[Theorem 9.48]{kerr2016ergodic}) and the ergodic decomposition, fix an ergodic $\mu$ with $h = h_\mu(X,T,G)>0$. Since there is no finite orbit (by minimality of $(X,T,G)$), $\mu$ is non-atomic.  Choose $A$ with $\mu(A)=\delta$ small enough that $\mathcal P=\{A,X\setminus A\}$ satisfies $0<H_\mu(\mathcal P)<h$, and note $h_\mu(G,\mathcal P)\le H_\mu(\mathcal P)$. Put $\varepsilon=\tfrac12\min\{H_\mu(\mathcal P),\,h-h_\mu(G,\mathcal P)\}>0$. Using \cite[Proposition 3.4]{Hochman_determinism_top_dyn:2012}, there is a partition $\mathcal{Q}$ and a continuous function $f \colon X \to [0,1]$ which is constant almost surely on each atom of $\mathcal{Q}$ (with pairwise different values on the atoms),  such that $d_{\mathcal{R}}(\mathcal{Q},\mathcal{P})<\varepsilon$, where $d_{\mathcal{R}}$ is the Rokhlin metric. We may choose $\varepsilon$ small enough so that $f$ is non-constant. We have that $h_{\mu}(\mathcal{Q})<h_{\mu}(\mathcal{P}) + \varepsilon < h_{\mu}(X,T,G)$. 
Let $s\in [0,1]$ be  such that $\{f<s\}$ and $\{f>s\}$ are two nonempty open sets. So, for any $x,y\in X$ not in the same $Z(G)$-orbit, there exists $g\in G$ such that $T^{g}x \in \{f<s\}$ and $T^{g}y\in \{f>s\}$. Consider the factor associated with $f$, that is, the (generalized) $G$-subshift\footnote{a subshift  is a closed subset of the product space invariant by the shift maps $S^h \colon (z_g)_{g\in G} \mapsto (z_{h^{-1}g})_{g\in G} $} $\{(f(T^{g^{-1}} x))_{g\in G} \mid x \in X \} \subseteq (f(X))^G$  with the factor map $ x\mapsto  (f(T^{g^{-1}}x))_{g\in G}$. We have that any fiber is at most countable since any fiber is a subset of $\{ T^g x:g \in Z(G)\}$ for some $x \in X$. By construction,  the entropy of the subshift $X_f$ is $h_\mu(\mathcal Q)$, but since the fiber are all at most countable, its entropy  should  be $h_\mu(X,T,G)$: a contradiction. \end{proof}

\subsection{Examples of doubly minimal systems}

In this section, we provide examples of $\Z^d $ doubly minimal systems. Because for these groups $\Z^d = Z(\Z^d) = \textnormal{FC}(\Z^d)$, by~\Cref{rem:true_doubly_minimal} we can simply talk about doubly minimal systems.

A first example is obtained by placing  elements of the original $\Z$-subshift of King ~\cite{King_top_minimal_self_joinings:1990}  side by side in columns. More generally, this enables us to construct doubly minimal examples for any group with center a finitely generated free abelian group.
\begin{proposition}
\label{prop:pull_back}
    Let $(X,T)$ be an $n$-fold sTMSJ for a group $H$. If there is an epimorphism $\pi:G\to H$ such that $\pi(Z(G)) = Z(H)$, then the pullback action of $G$ on $X$ given by $S^gx = T^{\pi(g)}x$ is an $n$-fold sTMSJ.
\end{proposition}
\begin{proof}
    By the definition of the action, the $S$-orbits and the $T$-orbits are the same. Since $\pi:G\to H$ is onto and the $T$-action is minimal, the $S$-action is also minimal. Next, consider $n$ points $x_1, \ldots, x_n$ on different $Z(G)$-orbits for $S$. Because $\pi(Z(G)) = Z(H)$, the points also lie on different $Z(H)$-orbits for $T$. Therefore, the denseness of the orbit of the $n$-tuple in $X^n$ for the $S$ action follows from that of the $T$-action.
\end{proof}
Nevertheless, the systems obtained through the previous proposition are not faithful, indeed every element of the kernel  $\ker(\pi)$ acts trivially. \\

The next lemma enables us to construct another example of an $n$-fold sTMSJ, starting from a single example. The proof is direct.
\begin{lemma}
    Let $X$ be an $n$-fold sTMSJ and $\phi\in\Aut(G)$. Then, the action $S^{g}(x) = T^{\phi(g)}(x)$ is an $n$-fold sTMSJ.
\end{lemma}
The next example is less trivial, it provides doubly minimal faithful $\Z^d$-actions. 
\begin{theorem}\label{theo:KingExample}
    For any dimension $d \ge 2$, there exists a faithful doubly minimal action on the Cantor set for $\Z^d$.
\end{theorem}

The proof is based on Oprocha's version~\cite{Oprocha_double_minimal:2019} of King's original construction. We  focus on the  construction of a doubly minimal system on $\Z^2$, and then explain how the construction works for higher dimensions.

\begin{proof}
    Consider the three letter alphabet $A=\{\tt 0, \tt 1, \tt a\}$, where $\tt a$ is a special spacer. We will use the construction of \cite{Oprocha_double_minimal:2019}, we recall the construction. We inductively define a sequence of words $H_n^{(0)}$, $H_n^{(1)} \in A^*$ of the same length. Suppose that $H_{n-1}^{(0)}$, $H_{n-1}^{(1)}$ have already been defined and have the same length. Let $k_n$ be an even number such that $k_n > |H_{n-1}^{(0)}|(2  |H_{n-1}^{(0)}| +1)$. Denote  $U_{n-1}^{(k)}= H_{n-1}^{(1)}\tt a(H_{n-1}^{(0)})^kH_{n-1}^{(1)}\tt a$ for $k= 3,4, 5$. Then  define 
   \begin{align*}\label{eq:Oprocha}
       H_n^{(0)} &= (H_{n-1}^{(0)})^{k_n}U_{n-1}^{(3)}(H_{n-1}^{(1)}\tt a)^{k_n}U_{n-1}^{(4)}(H_{n-1}^{(0)})^{k_n}U_{n-1}^{(5)}(H_{n-1}^{(1)}\tt a)^{k_n}H_{n-1}^{(0)} \numberthis \\
       H_n^{(1)} &= (H_{n-1}^{(0)}H_{n-1}^{(1)}\tt a)^{k_n/2}U_{n-1}^{(3)}(H_{n-1}^{(0)}H_{n-1}^{(1)}\tt a)^{k_n/2}U_{n-1}^{(4)}(H_{n-1}^{(0)}H_{n-1}^{(1)}\tt a)^{k_n/2}U_{n-1}^{(5)}(H_{n-1}^{(0)}H_{n-1}^{(1)}\tt a)^{k_n/2}H_{n-1}^{(1)}.
   \end{align*}
For a suitable $k_n$, the following noticeable facts holds: the two words have the same length $|H_n^{(0)}| = |H_n^{(1)}| = h_n$ and each word $H_n^{(0)}$ and $ H_n^{(1)}\tt a$ are concatenations of words $H_{n-1}^{(0)}$ and  $ H_{n-1}^{(1)}\tt a$. Moreover, Oprocha proved that shifted  words $H_n^{(0)}$, $H_n^{(1)}$ have coincidence as detailed in the next lemma.
\begin{lemma}{{\cite[Lemma 4.5]{Oprocha_double_minimal:2019}}}\label{lem:Oprocha} 
Let $\tt b$ be an auxiliary symbol $\tt b \notin \{\tt 0,\tt 1,\tt a\}$ and let $k\in \N^*$ be an integer.  

If $ k \le h_{n}/2$,  then for any  $\alpha \in \{0,1\}$, the words  $u= H_{n}^{(\alpha)}\tt b^k$
 and $w=  \tt b^kH_{n}^{(1-\alpha)}$ have a coincidence: there exists $k \le i \le  h_{n}-h_{n-1}$ such that $u_{[i,i+h_{n-1})}= w_{[i,i+h_{n-1} )}= H_{n-1}^{(0)}$.
\end{lemma}

We construct now a doubly minimal  $\Z^2$-subshift. We first need to introduce some notions. \par 
A {\it block} will be a coloration of some rectangle, i.e. an element of  $A^{[0,  h)\times [0,h')}$ for some $h,h'> 0$. Let $\b e_1$, $\b e_2$ denote the canonical base. A {\it concatenation} of two blocks $\b H \in A^{[0,  h_1)\times [0,h)} $ and $\b K \in  A^{[0,  k_1)\times [0,h)}$ {\it in the direction} $\b e_1$ is the block denoted $\b H \b K \in A^{[0,  h_1+k_1)\times [0,h)}$ defined as 
\[ \b H\b K|_{[0,  h_1)\times [0,h) } = \b H \textrm{ and } \b H\b K ( j ) = \b K(j-h_1 \b e_1) \textrm{ for } j \in {[h_1,  h_1+k_1)\times [0,h) }.
\]
One defines in a similar way the concatenation of blocks in the direction $\b e_2$.  

In the same way as in the one dimensional case, we construct inductively blocks on $\Z^2$ in two steps. We will concatenate block in the $\b e_1$ direction according to the order in \eqref{eq:Oprocha}. Then we will concatenate the obtained blocks in the $\b e_2$ direction according to the same order.  
Assuming $\b H_{n-1}^{(0)}$ and $\b H_{n-1}^{(1)}$ are blocks of $A^{[0,h_{n-1})^2}$, then  $\b H_{n}^{(0, \b e_1)}$ and $\b H_{n}^{(1, \b e_1)}$  will be concatenations of  blocks in the direction $\b e_1$, in the same order as in \eqref{eq:Oprocha}, that is
\begin{align*}\label{eq:Oprochae1}
       \b H_n^{(0, \b e_1)} &= (\b H_{n-1}^{(0)})^{k_n} \b U_{n-1}^{(3, \b e_1)}(\b H_{n-1}^{(1)}\tt a_{n-1}^{\b e_1})^{k_n}\b U_{n-1}^{(4, \b e_1)}(\b H_{n-1}^{(0)})^{k_n}\b U_{n-1}^{(5, \b e_1)}(\b H_{n-1}^{(1)}\tt a_{n-1}^{\b e_1})^{k_n}\b H_{n-1}^{(0)} \\
       \b H_n^{(1, \b e_1)} &= (\b H_{n-1}^{(0)}\b H_{n-1}^{(1)}\tt a_{n-1}^{\b e_1})^{k_n/2}\b U_{n-1}^{(3, \b e_1)}(\b H_{n-1}^{(0)}\b H_{n-1}^{(1)}\tt a_{n-1}^{\b e_1})^{k_n/2}\b U_{n-1}^{(4, \b e_1)}(\b H_{n-1}^{(0)}\b H_{n-1}^{(1)}\tt a_{n-1}^{\b e_1})^{k_n/2} \numberthis\\
        & \hspace{8cm}\b U_{n-1}^{(5, \b e_1)}(\b H_{n-1}^{(0)}\b H_{n-1}^{(1)}\tt a_{n-1}^{\b e_1})^{k_n/2}\b H_{n-1}^{(1)},
   \end{align*}
where $k_n$ is an even number such that $k_n > h_{n-1}(2  h_{n-1} +1)$; 
$\tt a_{n-1}^{\b e_1}\in A^{\{0\} \times [0, h_{n-1})}$ denotes the coloration with the single color $\tt a$   and  $\b U_{n-1}^{(k, \b e_1)}= \b H_{n-1}^{(1)}\tt a_{n-1}^{\b e_1}(\b H_{n-1}^{(0)})^k\b H_{n-1}^{(1)}\tt a_{n-1}^{\b e_1}$ for $k=3,4, 5$.
It is worth noting, from the properties of $H_n^{(0)}, H_n^{(1)}$ that $\b H_n^{(0,\b e_1)}$ and $\b H_n^{(1,\b e_1)}$ are colorations of the same rectangle $[0, h_n) \times [0, h_{n-1})$ and each block  $\b H_n^{(0, \b e_1)}$ and $\b  H_n^{(1, \b e_1)}\tt a_{n-1}^{\b e_1}$ are concatenations in the direction $\b e_1$ of blocks $\b H_{n-1}^{(0)}$ and  $\b H_{n-1}^{(1)}\tt a_{n-1}^{\b e_1}$,  $\b H_n^{(0, \b e_1)}$ is always followed by   $\b H_n^{(0, \b e_1)}$ or  $\b H_n^{(1, \b e_1)}$.

\begin{figure}[ht!]
\centering
    \begin{tikzpicture}[x=20pt,y=20pt]
\definecolor{azul}{rgb}{0.29,0.56,0.89}
\definecolor{rojo}{rgb}{0.82,0.01,0.11}
\definecolor{orange}{rgb}{0.96,0.65,0.14}
\definecolor{verde}{RGB}{126,211,33}

    
        \draw (-0.5,1) node {$\b H_{n}^{(0,e_1)} =$};

        \draw (1,2) -- (3,2) -- (3,0) -- (1,0) -- cycle;
        \draw (2,1) node {$\b H_{n-1}^{(0)}$};
        \draw [<->] (1,-0.2) -- (3,-0.2);
        \draw (2,-0.6) node {$h_{n-1}$};

        \draw (3.5,1) node {$\dots$};

        \draw (4,2) -- (6,2) -- (6,0) -- (4,0) -- cycle;
        \draw (5,1) node {$\b H_{n-1}^{(0)}$};
        \draw [decorate, decoration={brace}] (1,2.2) -- (6,2.2);
        \draw (3.5, 2.6) node {$k_n \textnormal{ times}$};

        \draw (6,2) -- (8,2) -- (8,0) -- (6,0) -- cycle;
        \draw (7,1) node {$\b H_{n-1}^{(1)}$};

        \draw [draw opacity = 0, fill = rojo] (8,2) -- (8.5,2) -- (8.5,0) -- (8,0) -- cycle;

        \foreach \x in {0,1,2}{
        \draw (8.5 + 2*\x,2) -- (10.5 + 2*\x,2) -- (10.5 + 2*\x,0) -- (8.5 + 2*\x,0) -- cycle;
        \draw (9.5 + 2*\x,1) node {$\b H_{n-1}^{(0)}$};
        }

        \draw (14.5,2) -- (16.5,2) -- (16.5,0) -- (14.5,0) -- cycle;
        \draw (15.5,1) node {$\b H_{n-1}^{(1)}$};

        \draw [draw opacity = 0, fill = rojo] (16.5,2) -- (17,2) -- (17,0) -- (16.5,0) -- cycle;
        
        \draw [decorate, decoration={brace, mirror}] (6,-0.2) -- (17,-0.2);
        \draw (11.5,-0.8) node {$\b U^{(3,e_1)}_{n-1}$};

        \draw (17.5,1) node {$\dots$};

        \draw (18,2) -- (20,2) -- (20,0) -- (18,0) -- cycle;
        \draw (19,1) node {$\b H_{n-1}^{(0)}$};

        \draw [<->] (20.2,2) -- (20.2, 0);
        \draw (21, 1) node {$h_{n-1}$};

    
        \draw (-0.5,-4) node {$\b H_{n}^{(1,e_1)} =$};

        \draw (1,-3) -- (3,-3) -- (3,-5) -- (1,-5) -- cycle;
        \draw (2,-4) node {$\b H_{n-1}^{(0)}$};
        \draw [<->] (1,-5.2) -- (3,-5.2);
        \draw (2,-5.6) node {$h_{n-1}$};

        \draw (3,-3) -- (5,-3) -- (5,-5) -- (3,-5) -- cycle;
        \draw (4,-4) node {$\b H_{n-1}^{(1)}$};

        \draw [draw opacity = 0, fill = rojo] (5,-3) -- (5.5,-3) -- (5.5,-5) -- (5,-5) -- cycle;

        \draw (6,-4) node {$\dots$};

        \draw (6.5,-3) -- (8.5,-3) -- (8.5,-5) -- (6.5,-5) -- cycle;
        \draw (7.5,-4) node {$\b H_{n-1}^{(0)}$};
        \draw (10.5,-3) -- (8.5,-3) -- (8.5,-5) -- (10.5,-5) -- cycle;
        \draw (9.5,-4) node {$\b H_{n-1}^{(1)}$};
        
        \draw [draw opacity = 0, fill = rojo] (10.5,-3) -- (11,-3) -- (11,-5) -- (10.5,-5) -- cycle;

         \draw [decorate, decoration={brace}] (1,-2.8) -- (11,-2.8);
         \draw (6, -2.4) node {$k_n/2 \textnormal{ times}$};

         \draw (11,-3) -- (13,-3) -- (13,-5) -- (11,-5) -- cycle;
        \draw (12,-4) node {$\b H_{n-1}^{(1)}$};
        \draw [draw opacity = 0, fill = rojo] (13,-3) -- (13.5,-3) -- (13.5,-5) -- (13,-5) -- cycle;

        \draw (15.5,-3) -- (13.5,-3) -- (13.5,-5) -- (15.5,-5) -- cycle;
        \draw (14.5,-4) node {$\b H_{n-1}^{(0)}$};

       \draw (16.8,-4) node {$\dots$};

       \draw (18,-3) -- (20,-3) -- (20,-5) -- (18,-5) -- cycle;
        \draw (19,-4) node {$\b H_{n-1}^{(1)}$};

        \draw [<->] (20.2,-3) -- (20.2, -5);
        \draw (21, -4) node {$h_{n-1}$};
    
\end{tikzpicture}
    \label{fig:op_hor}
    \caption{For the first step, the square blocks of the previous iteration, namely $\b H_{n-1}^{(0)}$ and $\b H_{n-1}^{(1)}$, are arranged following equation~\ref{eq:Oprochae1} where the blocks $\tt a^{e_1}_{n-1}$ are represented as red blocks.}
\end{figure}
In a second step, we proceed similarly in the direction $\b e_2$ but starting from the blocks $\b H_n^{(0,\b e_1)}, \b H_n^{(1,\b e_1)}$ instead of  $\b H_{n-1}^{(0)}, \b H_{n-1}^{(1)}$ and by concatenating them in the $\b  e_2$ direction, still in the same order as \eqref{eq:Oprocha}. Formally, the   blocks $\b H_n^{(0,\b e_2)}$ and $\b H_n^{(1,\b e_2)}$ are given by formulas \eqref{eq:Oprochae1} with  $\b H_n^{(i,\b e_1)}$ instead of $\b H_{n-1}^{(i)}$, $i=0,1$, 
$\b e_2$ instead of $\b e_1$ and where $\tt a_{n-1}^{\b e_2}\in A^{[0, h_{n}) \times \{0\}}$ denotes the coloration with the single color $\tt a$. Thanks to the properties of $H_n^{(j)}$, $j=0,1$, it is direct to check that the blocks $\b H_n^{(i,\b e_2)}$, $i= 0,1$ are colorations of the square $[0, h_n)^2$. Moreover $\b H_n^{(0,\b e_2)}$,  $\b  H_n^{(1, \b e_2)}\tt a_{n-1}^{\b e_2}$ are concatenations in the direction $\b e_2$ of blocks $\b H_{n-1}^{(0, \b e_2)}$ and  $\b H_{n-1}^{(1, \b e_2)}\tt a_{n-1}^{\b e_2}$. The block $\b H_{n}^{(0, \b e_2)}$ is always followed, in the direction $\b e_2$, by  $\b H_{n}^{(0, \b e_2)}$ or  $\b H_{n}^{(1, \b e_2)}$. 

\begin{figure}[H]
\centering
    \begin{tikzpicture}[x=22pt,y=20pt]
\definecolor{azul}{rgb}{0.29,0.56,0.89}
\definecolor{rojo}{rgb}{0.82,0.01,0.11}
\definecolor{orange}{rgb}{0.96,0.65,0.14}
\definecolor{verde}{RGB}{126,211,33}

\draw (-0.2,5.25) node {$\b H_{n}^{(0,e_2)} = $};

\draw (1,0) -- (11,0) -- (11,1) -- (1,1) -- cycle;
\draw (6,0.5) node {$\b H_{n}^{(0,e_1)}$};

        \draw [<->] (1,-0.2) -- (11,-0.2);
        \draw (6,-0.7) node {$h_{n}$};

        \draw [<->] (11.2,1) -- (11.2,0);
        \draw (11.9,0.5) node {$h_{n-1}$};

\draw (6,1.6) node {$\vdots$};

\draw (1,2) -- (11,2) -- (11,3) -- (1,3) -- cycle;
\draw (6,2.5) node {$\b H_{n}^{(0,e_1)}$};


\draw (1,3) -- (11,3) -- (11,4) -- (1,4) -- cycle;
\draw (6,3.5) node {$\b H_{n}^{(1,e_1)}$};

\draw [draw opacity=0, fill=rojo] (1,4) -- (11,4) -- (11,4.25) -- (1,4.25) -- cycle;

\draw (1,4.25) -- (11,4.25) -- (11,5.25) -- (1,5.25) -- cycle;
\draw (6,4.75) node {$\b H_{n}^{(0,e_1)}$};
\draw (1,5.25) -- (11,5.25) -- (11,6.25) -- (1,6.25) -- cycle;
\draw (6,5.75) node {$\b H_{n}^{(0,e_1)}$};

\draw (1,6.25) -- (11,6.25) -- (11,7.25) -- (1,7.25) -- cycle;
\draw (6,6.75) node {$\b H_{n}^{(0,e_1)}$};

\draw (1,7.25) -- (11,7.25) -- (11,8.25) -- (1,8.25) -- cycle;
\draw (6,7.75) node {$\b H_{n}^{(1,e_1)}$};

\draw [draw opacity=0, fill=rojo] (1,8.25) -- (11,8.25) -- (11,8.5) -- (1,8.5) -- cycle;

\draw (6,9.1) node {$\vdots$};

\draw (1,9.5) -- (11,9.5) -- (11,10.5) -- (1,10.5) -- cycle;
\draw (6,10) node {$\b H_{n}^{(0,e_1)}$};

    
\draw (12.8,5.25) node {$\b H_{n}^{(1,e_2)} = $};

\draw (14,0) -- (24,0) -- (24,1) -- (14,1) -- cycle;
\draw (19,0.5) node {$\b H_{n}^{(0,e_1)}$};

\draw (14,1) -- (24,1) -- (24,2) -- (14,2) -- cycle;
\draw (19,1.5) node {$\b H_{n}^{(1,e_1)}$};

\draw [draw opacity=0, fill=rojo] (14,2) -- (24,2) -- (24,2.25) -- (14,2.25) -- cycle;

\draw (19,2.85) node {$\vdots$};

\draw (14,3.25) -- (24,3.25) -- (24,4.25) -- (14,4.25) -- cycle;
\draw (19,3.75) node {$\b H_{n}^{(0,e_1)}$};

\draw (14,4.25) -- (24,4.25) -- (24,5.25) -- (14,5.25) -- cycle;
\draw (19,4.75) node {$\b H_{n}^{(1,e_1)}$};

\draw [draw opacity=0, fill=rojo] (14,5.25) -- (24,5.25) -- (24,5.5) -- (14,5.5) -- cycle;

\draw (14,5.5) -- (24,5.5) -- (24,6.5) -- (14,6.5) -- cycle;
\draw (19,6) node {$\b H_{n}^{(1,e_1)}$};

\draw [draw opacity=0, fill=rojo] (14,6.5) -- (24,6.5) -- (24,6.75) -- (14,6.75) -- cycle;

\draw (14,6.75) -- (24,6.75) -- (24,7.75) -- (14,7.75) -- cycle;
\draw (19,7.25) node {$\b H_{n}^{(0,e_1)}$};

\draw (19,8.75) node {$\vdots$};

\draw (14,9.5) -- (24,9.5) -- (24,10.5) -- (14,10.5) -- cycle;
\draw (19,10) node {$\b H_{n}^{(1,e_1)}$};
    
\end{tikzpicture}
    \label{fig:op_ver}
    \caption{For the second step, we obtain new square blocks by concatenating the rectangles $\b H_{n}^{(i,e_1)}$ following an analog of equation~\ref{eq:Oprocha}, where once again the blocks $\tt a^{e_2}_{n-1}$ are represented as red blocks.}
\end{figure}
\noindent Finally, we set $\b H_n^{(0)} = \b H_n^{(0, \b e_2)}$ and  $\b H_n^{(1)} = \b H_n^{(1, \b e_2)}$.

Since the blocks are obtained by concatenating blocks in the same order as in the one dimensional case (see Equation \eqref{eq:Oprocha}), we get similar combinatorial properties as in Lemma \ref{lem:Oprocha}.   
\begin{lemma}\label{lem:Oprocha2D}
Let  $n \ge t, v >2$ be  integers and  $k= (k_{1}, k_{2}) \in \Z^2$  be a vector.
\begin{enumerate}[label=(\roman*)]
	\item     If  $\|k\|_\infty = \max(|k_{1}|, |k_{2}|)  \le h_{n-1}/2$, then the blocks $\b H_n^{(\alpha)}$ and  $T^{-k}\b  H_n^{(1-\alpha)}$, $\alpha= 0,1$, have a coincidence outside the ball of radius $\|k\|_\infty$. More precisely, there  is a position $i \in \Z^2$, $\|k\|_\infty \le \|i\|_\infty \le h_n-h_{n-1} $  such that
     $\b H_n^{(\alpha)}|_{i+ [0,h_{n-1})^2} = \b H_n^{(1-\alpha)}|_{i+k +[0,h_{n-1})^2} = \b H_{n-1}^{(0)}$.

  \item If $h_{t-1}/2 <k_{1} \le  h_{t}/2$  and  $h_{v-1}/2 <k_{2} \le  h_{v}/2$, with $t,v \le n-2$ then the blocks $\b  H_{n}^{\alpha}$ and $T^{-k} \b  H_{n}^{\beta}$, $\alpha, \beta \in \{0,1\}$ have a coincidence. More precisely, denoting $u= \max (t, v)$, there is a position $q\in \Z^{2}$ such that $q+ [0, h_{u})^{2}$ and   $q+k+ [0, h_{u})^{2}$ are subset of $[0, h_{n})^{2}$ and 
  $\b H_{n}^{(\alpha)}|_{q+ [0, h_{u})^{2}}=     \b H_{n}^{(\beta)}|_{q+k+ [0, h_{u})^{2}} = \b H_{{u}}^{(0)}$.
\end{enumerate}    
\end{lemma}
\begin{proof} Since the blocks are concatenations in the $\b e_2$ direction of  blocks $\b H_n^{(0, \b e_1)}$ and $\b H_n^{(1, \b e_1)}\tt a^{\b e_2}_{n-1}$ ordered as the words $H_n^{(0)}$, $H_n^{(1)}$, by Lemma \ref{lem:Oprocha} (i), there is an integer $i'_2 \ge |k_2|$ such that $\b H_n^{(\alpha)}|_{(0,i'_2) + [0,h_{n}) \times [0,h_{n-1})} = \b H_n^{(1-\alpha)}|_{(k_1,k_2+i'_2) + [0,h_{n}) \times [0,h_{n-1})} = \b H_n^{(0, \b e_1)}$. 
Notice that this level do not correspond to the top row of the block: precisely $i'_2 + k_2 +h_{n-1} < k_2 + h_n$, because the condition $\|k\|_\infty  \le h_{n-1}/2$ would impose  on the top rows of $\b H_n^{(0)} $ and $\b H_n^{(1)} $ to be equal.

\noindent Since the block $\b H_n^{(0, \b e_1)}$ is always followed, in the $\b e_2$ direction, by $\b H_n^{(0, \b e_1)}$ or  $\b H_n^{(1, \b e_1)}$,   
there is an integer  $i_2 \ge i'_2$ such that $\b H_n^{(\alpha)}|_{(0,i_2) + [0,h_{n}) \times [0,h_{n-1})} =\b H_n^{(\beta, \b e_1)}$  and  \[\b H_n^{(1-\alpha)}|_{(k_1,i_2+k_2) + [0,h_{n}) \times [0,h_{n-1})} = \b H_n^{(1-\beta, \b e_1)},\] for some $\beta \in \{0,1\}$.
Again,  Lemma \ref{lem:Oprocha} (i) applied to $\b H_n^{(\beta, \b e_1)}$ and $T^{-(k_{1},0)}\b H_{n}^{(1-\beta, \b e_{1})}$, gives the existence of an integer $i_1 \ge |k_1|$ such that $i =(i_1, i_2)$ satisfies the statement of Item (\textit{i}).

To prove Item (\textit{ii}), let us consider the case where  $v \ge t$. The other case is similar. The situation $\alpha\neq \beta$ is treated in item (\textit{i}), so we assume $\alpha=\beta$.
Recall that the left bottom block of $\b H_{n}^{(\alpha)}$ is of the form $\b H_{v+1}^{(0)}$. 
Since this block is always followed, in the $\b e_1$ direction, by $\b H_{v+1}^{(0)}$ or  $\b H_{v+1}^{(1)}$,   
there is an integer  $i_2$ such that $\b H_n^{(\alpha)}|_{(0,i_2h_{v+1}) + [0,h_{v+1})^{2}} =\b H_{v+1}^{(1)}$  and  $\b H_{n}^{(\alpha)}|_{(0,(i_{2}-1)h_{v+1}) + [0,h_{v+1})^{2}} = \b H_{v+1}^{(0)}$.
By Item (\textit{i}), the blocks  $T^{-k} \b H_{v+1}^{(0)}$ and $\b H_{v+1}^{(1)}$ have a coincidence.  
\end{proof}

Let $x\in \{\tt 0,\tt  1,\tt  a \}^{\Z^2}$ be a point  such that for each $n>0$ $x|_{[-h_n, h_n)^2} (j)   = \b H_{n+1}^{(0)}|_{[0, 2h_n)^2} ( j+ h_n \b e_1 +h_n \b e_2)$, for $ j \in [-h_n, h_n)^2$.  This means that $ x|_{[-h_n, h_n)^2}$ is the double concatenation in the direction $\b e_2$ of the   double concatenation  in the direction $\b e_1$ of $\b H_n^{(0)}$. This block is actually a sub-block of the one defined at the level $n+1$ so that the blocks appearing in $x$ are the same as the ones appearing in $\b H_n^{(\alpha)}$, $\alpha =0,1$.

\begin{lemma}
The $\Z^2$-subshift $X= \overline{\textrm{orb}_T (x)}$ that is the orbit closure of $x$ is minimal and aperiodic, i.e the stabilizer of any point is trivial. 
\end{lemma}
\begin{proof}
The block  $x|_{[-h_n, h_n)^2}$ appears in the left bottom of each block $\b H_{n+1}^{(0)}$ which appears in each block  $\b H_m^{(0)}$ and  $\b H_m^{(1)}$, $m>n+1$. Hence $x|_{[-h_n, h_n)^2}$ occurs syndetically. The  orbit closure  $X$  for the shift map $T$ of $x$ is then a minimal subshift. 

The subshift being minimal, any potential period $v= (v_{1}, v_{2})\in \Z^2\setminus\{\bf 0\}$ is a period for $x$: $T^v x = x$.  Take $n$ large enough so that $h_{n-2}/2>\|v\|_\infty$. In particular this means that for each  integers $k\in \Z$, $m\in \N^{*}$, the patterns $T^{kv}\tt a_{m}^{\b e_{1}}$ and $T^{kv}\tt a_{m}^{\b e_{2}}$ occur in $x$. This imposes that there is a translated of a sublattice on which the restriction of $x$ is $\tt a$: there is a position $q \in \Z^{2}$ such that for any $(\ell^{1}, \ell^{2}) \in \Z^{2}$, $x_{q+(v_{1}\ell^{1}, v_{2} \ell^{2}) } = \tt a$.   So, this is also the case for the pattern $\b H_{m}^{(0)}$. Since the patterns   $\b H_{m}^{(0)}$ and $T^{-\b ke_{i}} \b H_{m}^{(0)}$, for  any  $m> n+2$, $ i= 1,2$ and $0 \le k \le \| v \|_{\infty}$, coincide on a pattern $\b H_{n}^{(0)}$ (Lemma \ref{lem:Oprocha2D}), the pattern $\b H_{n}^{(0)}$ is covered by all the translations of the lattice $(v_{1},0) \Z \oplus (0,v_{2})\Z$ where the coloration is constant to $\tt a$. Hence the coloration $x$ is constant, a contradiction.  
\end{proof}

\begin{lemma}
The system  $(X, T, \Z^{2})$ is doubly minimal.
\end{lemma}

\begin{proof}
To prove double minimality, it suffices to show that any pair of points $y,z$  in different orbits are proximal, i.e. the  distance   between $T^n y $  and $T^n z$ goes to zero for an infinite sequence of $n\in \Z^2$ \cite[p.747]{King_top_minimal_self_joinings:1990}.

Indeed, since $y$ is decomposed into blocks of the form
$\b H_n^{(0)} $ and $\b H_n^{(1)}$,  
for each integer $n>0$, there are integers $ - h_n/2 \le i_1(n), i_2(n) \le h_n/2$  such that $y_{(i_1,i_2)+ [0, h_n)^2} = \b H_n^{(\alpha_{n}, \b e_2)}$, for some $\alpha_{n} \in \{0, 1\}$. Similarly, there are integers $ i_\ell - h_{n-2}/2 \le j_\ell(n) \le i_\ell+h_{n-2}/2$, $\ell= 1,2$ such that $z_{(j_1,j_2)+ [0, h_{n-2})^2} = \b H_{n-2}^{(\beta_{n}, \b e_2)}$  with  $\beta_{n} \in \{ 0,1\}$.  In particular, the indices $i_\ell$ and $j_\ell$ are close in the following sense $|i_\ell-j_\ell| \le h_{n-2}/2$. 
Let $v_n$ and $t_n$ be the smallest integers such that  $(h_{v_{n}-1}/2 <)$ $|i_2-j_2| \le h_{v_n}/2$ and  $(h_{t_{n}-1}/2 <)$ $|i_1-j_1| \le h_{t_n}/2$.

Setting $u_{n} = \max (v_{n}, t_{n})$, Lemma \ref{lem:Oprocha2D} (\textit{ii}) provides the existence of one  position $q_{n} \in \Z^{2}$ such that
\begin{align}\label{eq:Oprocha2}
y|_{q_{n}+ [0, h_{u_n})^{2} } = z|_{q_{n}+ [0, h_{u_n})^{2}} = \b H_{u_{n}}^{(0)}.
\end{align}
In particular,  the elements $y$ and $z$ have coincidence.

Doing the former step for each integer $n$ provides integer sequences $(v_n)_{n\ge 0}$ and $(t_n)_{n \ge 0}$.  If  one of the two sequences is unbounded, the elements $y$ and $z$ coincide  on larger and larger set (of size the same size as $\b H_{\max (v_n, t_n)}^{(0)}$ by equation \eqref{eq:Oprocha2}). Hence the two elements are proximal.

In contrast, let us consider the case where the two sequences $(v_n)_n$ and $(t_n)_n$ are bounded, Then, the positions $i(n) = (i_{1}, i_{2})$ and $j(n)= (j_{1}, j_{2})$ belong  to a bounded set. So, up to taking a subsequence of $n$, we can assume that the positions $i$ and $j$ are independent of $n$. 
We claim that the points $y$ and $z$ are proximal or either in the same orbit. To prove it, let us assume they are not proximal. This means they can not coincide on sufficiently large ball. So, there is some integer $N_1 >0$ such that $y$ and $z$ cannot coincide on any square of size $h_n$ for $n >N_1$. Lemma \ref{lem:Oprocha2D} (\textit{i}) then ensures that $y|_{i+[0, h_n)^2} = z|_{j+[0, h_n)^2} = \b H_n^{(\alpha_n)}$ for the same $\alpha_n \in \{0,1\}$ for any $n>N_1+1$. Also, the same lemma gives that $y|_{i+ (\epsilon_1  h_n, \epsilon_2  h_n)+[0, h_n)^2} = z|_{j+(\epsilon_1  h_n, \epsilon_2  h_n)+[0, h_n)^2}$ for any $\epsilon_1, \epsilon_2 \in \{-1, 0, 1\}$. Since it is true for infinitely many $n$, this provides that $T^{-i}y= T^{-j}z$ meaning that $y$ and $z$ are in the same orbit. 
\end{proof}

The construction of a $\Z^d$ subshift that is doubly minimal follows the same strategy as in dimension $2$. By concatenating blocks iteratively in each direction $\b e_{2}, \ldots, \b e_{d}$ according to the same order as \eqref{eq:Oprocha}.  The properties and the proofs are the same. 

\end{proof}

\subsection{Limits to $n$-folding}\label{sec:LimitnFolding}

As mentioned in the introduction, in~\cite{King_top_minimal_self_joinings:1990} King showed that no $\Z$-system can have a 4-fold sTMSJ. Our goal in this section is to generalize his result to other groups. We will first establish a general theorem that bounds the number of folds if the group allows a certain decomposition. Next, we move towards a generalization of King's proof strategy that involves the use of asymptotic pairs.

We begin with the following preliminary lemma, which will inform how the general proof works. We say that $(x,y)$ is a $\varepsilon$-bi-asymptotic pair if $x\neq y$ and $d(T^gx,T^gy)\leq \ve$ for all $g\in G$ except for a finite set. Furthermore, we introduce $\|\phi\| = \sup_{x\in X}d(\phi(x), x)$ for $\phi\in\Aut(X,T,G)$.

\begin{lemma}
    Let $(X,T,G)$ be a topological dynamical system that contains a $\varepsilon$-bi-asymptotic pair for every $\varepsilon>0$. Then $(X,T,G)$ is not strong doubly minimal.
\end{lemma}

\begin{proof}
 Assume that $(X,T,G)$ is strong doubly minimal. As before, thanks to  \cref{lem:uncountable_aut,lem:coalescent}, take $\epsilon>0$ such that if $\phi \in \Aut(X,T,G)$ is such that $\|\phi\|\leq \epsilon$, then $\phi={\rm id}$. Let $(x,y)$ be a $\epsilon$-bi-asymptotic pair. It is clear that the orbit of $(x,y)$ is not $X\times X$, hence $y=T^g x$ for some $g\in Z(G)$. But this implies that $\|T^g\|\leq \epsilon$, and therefore $T^g={\rm id}$, contradicting that $x\neq y$.  
\end{proof}

\begin{lemma}
\label{lem:diff_orbits}
    Let $(X,T,G)$ be a minimal topological dynamical system and assume that $(G,\rho)$ has thick horoballs. Suppose that for every $\ve>0$ there exists $h\in\partial G$ with an $(h,\ve)$-asymptotic pair. Then, for all $\ve>0$ there exists $h\in\partial G$ and an $(h,\ve)$-asymptotic pair $(x,y)$ such that $x$ and $y$ lie on different $Z(G)$ orbits. 
\end{lemma}

\begin{proof}
Suppose there exists $\ve_0>0$ such that $\|T^g\|> \ve_0$ for all $g\in Z(G)\setminus\{{\rm Id}\}$. Let $0<\ve\leq \ve_0$. Consider $h\in\partial G$, and an $(h,\ve)$-asymptotic pair $(x,y)$ with $x\neq y$. Suppose there exists $g_0\in Z(G)$ such that $y = T^{g_0}x$. Then, $d(T^gx,T^{g_0g}x)\leq \varepsilon$ for all $g\in\{h<0\}$. Because $\{h<0\}$ is thick, and the action is minimal, by~\cref{lem:minimal_por_horobola} $\{T^gx \mid g\in\{h<0\}\}$ is dense. By continuity, $\|T^{g_0}\|\leq\ve\leq\ve_0$. This implies $g_0 = 1_G$, which in turn implies $x = y$. Therefore, $x$ and $y$ lie in different $Z(G)$ orbits.

So, we may assume that there exists $g_n \in Z(G)$ with $\|T^{g_n}\|\to0$. By \cref{lem:uncountable_aut}, there exists $\phi\notin\{T^g \mid g\in Z(G)\}$ with $\sup_{x\in X} d(\phi(x),x) \leq \ve$. In particular, for any $x \in X$ we have $\phi(x)\notin Z(G)x$. Then $(x,\phi(x))$ are $\ve$-asymptotic (for any $h$). 
\end{proof}

A collection of subsets $\Omega_1,\ldots,\Omega_d$ is thickly full if each $\Omega_i$ is thick and the set $G\setminus \bigcup \Omega_i$ is finite.

\begin{theorem} \label{thm:king_general}
Let $(X,T,G)$ be a topological dynamical system. Assume that there exists $N\in \N$ such that for every $\varepsilon>0$ there exist points $x_i,y_i\in X$ for $i\in\{1,\ldots,N\}$ and sets $\Omega_i^{(\varepsilon)}\subseteq G$ such that 
\begin{itemize}
    \item The points $x_i$ and $y_i$ are distinct, and $d(T^gx_i,T^gy_i)\leq \ve$ for all $g\in \Omega_i^{(\varepsilon)}$.
    \item The collection $\Omega_1^{(\varepsilon)},\ldots,\Omega_N^{(\varepsilon)}$ is thickly full.
\end{itemize}

Then $(X,T,G)$ does not have $2N$-fold sTMSJ.
\end{theorem}

\begin{proof}
Note that the assumptions imply that $X$ is infinite. Assume that $(X,T,G)$ does have $2N$-fold sTMSJ. Let $\ve_0>0$ be such that $\|T^g\|> \ve_0$ for all $g\in Z(G)\setminus\{1_G\}$. For $0<\ve\leq \ve_0$, let $(x_i,y_i)$, and $\Omega_i$, $i=1,\ldots,N$ given as in the statement of \cref{thm:king_general}. Then we must have that $x_i$ and $y_i$ lie in different $Z(G)$ orbits. Indeed, if $y_i=T^{g}x_i$ for some $g\in Z(G)$, by minimality of $(X,T,G)$ and thickness of $\Omega_i$, we would get $\|T^g\|\leq \ve$, and hence $T^g={\rm id}$, contradicting that $x_i\neq y_i$. 
Let $A=\{z_1,\ldots,z_k\}$ be a maximal subset of $\{x_1,y_1,\ldots,x_N,y_N\}$ with the property that all of its elements lie in different $Z(G)$-orbits. Since $x_i,y_i$ do not lie in the same $Z(G)$-orbit,  we have $2\leq k \leq 2N$. 
For $i=1,\ldots,N$, set $h_i,k_i\in Z(G)$ such that $T^{h_i}x_i, T^{k_i}y_i\in A$ (these exist by maximality of $A$) and write $T^{h_i}x_i=z_{a(i)}$, $T^{k_i}y_i=z_{b(i)}$, for $a(i),b(i)\in \{1,\ldots,k\}$. Note that for all $1\leq i\leq N$, $a(i)\neq b(i)$ since $x_i$ and $y_i$ lie in different $Z(G)$-orbits.  
Consider the tuple $(z_1,\ldots,z_k)$. By construction, all coordinates of this tuple lie in different $Z(G)$-orbits, and since $(X,T,G)$ has $2N$-fold sTMSJ (and in particular has $k$-fold sTMSJ), the orbit of $(z_1,\ldots,z_k)$ is dense in $X^k$. Let $(w_1,\ldots,w_k) \in X^k$, and let $(g_j)_{j\in \N}$ be a sequence in $G$ such that $T^{g_j}z_i\to w_i$ as $j\to \infty$, for all $i=1,\ldots,k$. Taking a subsequence if needed, we may assume that $g_j\in \Omega_t$ for some $t\in \{1,\ldots,N\}$. This implies that  $d(T^{g_j}x_t,T^{g_j}y_t)\leq \varepsilon$ for all $j$ and $(T^{g_j}T^{h_t}x_t,T^{g_j}T^{k_t}y_t)=(T^{g_j}z_{a(t)}, T^{g_j}z_{b(t)})\to (w_{a(t)},w_{b(t)})$.  We conclude that $w_{b(t)} \in T^{k_t}B(T^{h_t^{-1}}w_{a(t)},\ve)$. Recalling that $a(t)\neq b(t)$ we obtain:
 \begin{equation}
  \tag{P}\label{eq:property_asympt}
  \parbox{\dimexpr\linewidth-7em}{
  For any $(w_1,\ldots,w_k) \in X^k$, there exist $1\leq a<b\leq k$ such that \[w_b\in \bigcup_{1\leq t\leq N} T^{k_t}B(T^{h_t^{-1}}w_a,\ve) \cup T^{h_t}B(T^{k_t^{-1}}w_{a},\ve)\]
  }
\end{equation}

\begin{claim}
Fix $L \in \N$. For any small enough $\ve>0$, any finite set of points $w_1,\ldots,w_L \in X$ and $h_1,\ldots,h_{L}\in Z(G)$, the set 
$\bigcup_{1\leq i\leq L} T^{h_i^{-1}}B(w_i,\ve)$ is not equal to $X$. 
\end{claim}
\begin{proof}[Proof of the claim]
Assume for the sake of contradiction that the conclusion is not true. Take $\ve_n\to 0$,  $w_{1,n},\ldots,w_{L,n} \in X$, $h_{1,n},\ldots,h_{L,n}\in Z(G)$ such that 
\[ X=\bigcup_{1\leq i\leq L} T^{h_{i,n}^{-1}}B(w_{i,n},\ve_n)\]
Note that this implies that $Z(G)$ is infinite, otherwise the equicontinuity of the family $\{T^{h}\}_{h\in Z(G)}$ would imply that $X$ is finite.
As $Z(G)$ is infinite, choose $h^{(1)},...,h^{(L+1)}\in Z(G)$ such that the corresponding maps $T^{h^{(1)}},...,T^{h^{(L+1)}}$ are distinct. Fix $x\in X$. Then, for every $n\in\N$ by the pigeonhole principle, there exists $q\neq p$ and $i_n$ such that $T^{h^{(p)}}x$ and $T^{h^{(q)}}x$ belong to $T^{h_{i_n,n}^{-1}}B(w_{i_n,n},\ve_n)$. Up to a subsequence, we suppose $p$ and $q$ do not depend on $n$. Then, if we denote $h = h^{(p)}$ and $h'=h^{(q)}$, we have that $h, h'\in Z(G)$. Then, $d(T^{h}T^{h_{i,n}}x,T^{h'}T^{h_{i,n}}x) \leq 2\ve_n$. Again, taking a subsequence if needed, we may assume $T^{h_{i_n,n}}x\to x'$ for some $x'\in X$, and so $T^{h}x'=T^{h'}x'$. By~\cref{lem:min+faithfull=>free}, this implies $T^{h}=T^{h'}$, which contradicts the choice of elements $h^{(j)}$.
\end{proof}

Given $\ve>0$ so that the claim above holds for $L=2kN$, we can construct a point $(w_1,\ldots,w_k)\in X^k$ for which $\eqref{eq:property_asympt}$ does not hold. Indeed, start with any $w_1$ and choose $w_2$ so that $\eqref{eq:property_asympt}$ does not hold for $w_b=w_2$ and $w_a=w_1$. Such $w_2$ exists thanks to the previous claim. Inductively, if we have chosen $w_1,\ldots,w_l$, we pick $w_{l+1}$ such that it is not in $ \bigcup_{i\leq l}\bigcup_{1\leq t\leq N} T^{k_t}B(T^{h_t^{-1}}w_i,\ve) \cup T^{h_t}B(T^{k_t^{-1}}w_{i},\ve)$. Thanks to the claim, such a $w_{l+1}$ always exists. Finishing with $l+1=k$, we find the announced point. Hence $(X,T,G)$ does not have $k$-fold (and hence $2N$-fold) sTMSJ. 
\end{proof}

\subsubsection{Limits to folding through horoballs} \label{sec:limits_fold_tmsj}
This section is devoted to the proofs on the limits of foldings for finitely generated abelian groups and discrete Heisenberg groups, namely the proofs to \Cref{thm:no-tmjs-Zd} and \cref{thm:no-tmjs-Heisenberg}, generalizing King's result for $\Z$. We begin by deriving an immediate consequence of~\cref{thm:king_general}. 
\begin{corollary}
Let $(X,T,G)$ be a topological dynamical system, where $G$ is endowed with a right-invariant and proper metric. Assume that there exist a finite set $K\subseteq G$ and $N\in \N$ such that for any $\ve>0$, there exist $h_1,\ldots,h_N\in \ND_{\ve}(X)$, with $\{h_i<0\}$ thick, and $g_{i,k}\in Z(G)$, for $1\leq i\leq N$, $k\in K$, such that 
\[\bigcup_{k\in K} \bigcup_{1\leq i\leq N} kg_{i,k}\{h_i<0\} =G\]
Then, $(X,T,G)$ has no $2|K|N$ sTMSJ.
\end{corollary}

\begin{proof}
Consider $\ve>0$. Because $K$ is finite, there exists $0<\delta\leq \ve$ such that $d(x,y)\leq \delta$ implies $d(T^kx,T^ky)\leq \ve$ for all $k\in K$. By applying the hypothesis to $\delta$, there exists $h_1,\ldots,h_N\in \ND_{\delta}(X)$, with $\{h_i<0\}$ thick, and $g_{i,k}\in Z(G)$, $1\leq i\leq N$, $k\in K$ such that 
\[\bigcup_{k\in K} \bigcup_{1\leq i\leq N} kg_{i,k}\{h_i<0\} =G.\]

Since $h_i\in\ND_{\delta}(X)$ for every $i$, there exist $x_i\neq y_i$ in $X$ with $d(T^gx_i,T^gy_i)\leq \delta$ for all $g\in\{h_i<0\}$. For each pair $(i,k)$ we define $\Omega_{i,k} = k g_{i,k}\{h_i<0\}$, and the set of distinct points $x_{i,k} =T^{g_{i,k}^{-1}}x_i$ and $y_{i,k} = T^{g_{i,k}^{-1}}y_i$. Take $f = kg_{i,k}g\in\Omega_{i,k}$. Because $g$ belongs to the center $f g_{i,k}^{-1} = kg$, and therefore $T^f x_{i,k} = T^{kg} x_i$ and $T^f y_{i,k} = T^{kg} y_i$. Thus, $d(T^fx_{i,k}, T^fy_{i,k})\leq\ve$.

The $|K|N$ sets $\Omega_{i,k}$ cover $G$. Since the elements $g_{i,k}$ belong to the center, we have $kg_{i,k}\{h_i<0\} = k\{h_i<0\}g_{i,k}$. Because right translation preserves thickness, it suffices to see that $k\{h_i<0\}$ is thick. Given $r > 0$, the ball $B_{r}(1_G)$ is finite, so $R = \max\{\rho(k^{-1}f'k,1_G) \mid f'\in B_r(1_G)\}$ is finite. Consider $g$ such that $B_R(g)\subseteq \{h_i<0\}$. If $\rho(f,kg) \le r$, then $f' = f(kg)^{-1}$ satisfies $f'\in B_r(1_G)$, and $k^{-1}fg^{-1} = k^{-1}f'k$. Then, $\rho(k^{-1}f, g) \le R$ and $f \in k\{h_i<0\}$. Thus $B_{r}(kg)\subseteq k\{h_i<0\}$ for every $r$. This implies the sets and points satisfy the hypothesis of~\Cref{thm:king_general}, and thus $X$ has no $2|K|N$-fold sTMSJ.

\end{proof}

\begin{proof}[Proof of \cref{thm:no-tmjs-Zd}]
    First, assume that the torsion subgroup of $G$, $Tor(G)$, is trivial, so that $G \simeq \Z^d$. 
    By Proposition~\ref{prop:descomp_Z} there exists $\{v_i\}_{i=1}^k\subseteq\ND(X)$ with $k\leq d+1$ such that
    \[\bigcup_{i=1}^k\{g\in\Z^d \mid \langle v_i, g\rangle\leq 0\} = \Z^d.\]
    Denote $\Omega_i=\{g\in \Z^d \mid \langle v_i, g\rangle \leq 0\}$ and $\Omega_i'=\{g\in \Z^d \mid \langle v_i, g\rangle < 0\}$. For each $i$, we can find $g_i\in \Z^d$ such that $\Omega_i+g_i\subseteq \Omega_i'$. Note that for $x,y\in X$ such that $d(T^gx,T^gy)\leq \ve$ for all $g\in \Omega_i'$, we have that $d(T^{g}(T^{g_i}x),T^{g}(T^{g_i}y))\leq \ve$ for all $g\in \Omega_i$. 
    
    Noting that the collection $\Omega_1,\ldots,\Omega_k$ is thickly full, we get that it fulfills the conditions of \cref{thm:king_general}, and we obtain the desired conclusion. 

   For the general case, as the action of $Tor(G)$ is equicontinuous, it is not hard to check that, for the $v_i$ as before, $\tilde{\Omega}_i=\Omega_i\oplus Tor(G)$  satisfies the hypothesis of \cref{thm:king_general}. The conclusion follows. 
\end{proof}

M. Hochman informed us of a more elementary proof of \cref{thm:no-tmjs-Zd}, though it does not provide explicit control of the $2(d+1)$-fold.

The following result can be proved using a similar idea. 
\begin{proof}[Proof of \cref{thm:no-tmjs-Heisenberg}]
    By Proposition~\ref{prop:descomp_H} there exists $\{(\vec{a}_i,\vec{b}_i,c_i)\}_{i=1}^k\subseteq\ND(X)$ with $k\leq 2d+1$ such that
    \[\bigcup_{i=1}^k\left\{(\vec{v},\vec{u},t)\in H_{2d+1}(\Z) \mid \langle V(\vec{a}_i,\vec{b}_i,c_i), (\vec{v},\vec{u})\rangle\leq 0\right\} = H_{2d+1}(\Z).\]
    The conclusion follows from \cref{thm:king_general}, using the same arguments as in the proof of \cref{thm:no-tmjs-Zd}.
\end{proof}

\subsubsection{Groups admitting infinite-fold sTMSJs}

Recall that a group $G$ admits an $\infty$-fold sTMSJ if there exists an action $G\act X$ that is an $n$-fold sTMSJ for all $n\in \N$. It is not hard to construct such groups. Let $X=\{0,1\}^{\Z}$. For any integer $n \ge 1$, let $W_n=\{0,1\}^{[-n, n]}$ be the set of all finite words of length $2n+1$ and let $\sigma$ be any permutation of $W_n$. Define a homeomorphism $\phi_{\sigma} \colon X \to X$ that permutes the central block as follows
$$
(\phi_{\sigma}(x))_j =
\begin{cases}
(\sigma(x|_{[-n, n]}))_j & \text{if } j \in [-n, n] \\
x_j & \text{if } j \notin [-n, n]
\end{cases}
$$
Let $G$ be the group generated by all such transformations for all possible block sizes and all possible permutations, i.e., $G = \langle \{ \phi_{\sigma} \mid n \ge 1, \sigma \in \text{Sym}(W_n) \} \rangle$. Clearly, the action $(X,T,G)$ has $n$-fold sTMSJ for any $n \ge 1$.

This example belongs to the class of groups that admit the recently introduced property of \emph{deeply transitive actions}~\cite{Kra_Schmieding_invariant_random_compact:2026}. These actions permute different leveled partitions of a Cantor space $X$ in a way analogous to the previous example (see~\cite[Definition 6.3]{Kra_Schmieding_invariant_random_compact:2026}). A short computation shows every such action has $\infty$-fold sTMSJ. Examples of such actions are the action of the stabilized automorphism group of the full-shift $\Aut^{(\infty)}(\{0,...,n-1\}^\Z, \sigma) = \bigcup_{k\geq 1}\Aut(\{0,...,n-1\}^\Z, \sigma^k)$ on the full-shift, and Thompson's $V$ group action on $\{0,1\}^\N$.\\

The more challenging question is to find an abelian infinite joining order group. According to the results in the previous section, this group cannot be finitely generated. In particular, we do not know if $\bigoplus_{i\in \N}\Z/2\Z$ or $\bigoplus_{i\in \N}\Z$ have $\infty$-fold sTMSJ.


\bibliographystyle{abbrv}
\bibliography{refs,alexandria}

\end{document}